\documentclass[12pt]{article} %***
\usepackage[sectionbib]{natbib}
\usepackage{graphicx}
\usepackage{epstopdf}
\usepackage{color}
\usepackage{url}
\usepackage[T1]{fontenc}
\usepackage{array,epsfig,fancyheadings,rotating}
\usepackage[]{hyperref}  %<----modified by Ivan
\usepackage{algorithm}
\usepackage{algorithmic}
\usepackage{sectsty, secdot}
\usepackage{multirow}

\sectionfont{\fontsize{12}{14pt plus.8pt minus .6pt}\selectfont}
\renewcommand{\theequation}{\thesection\arabic{equation}}
\subsectionfont{\fontsize{12}{14pt plus.8pt minus .6pt}\selectfont}
\usepackage{amsmath}
\usepackage{amssymb}
\usepackage{amsfonts}
\usepackage{multirow}
\usepackage{amsthm}

\newtheorem{theorem}{Theorem}
\newtheorem{lemma}{Lemma}
\newtheorem{corollary}{Corollary}

\theoremstyle{definition}

\newtheorem{assumption}{Assumption}
\newtheorem{example}{Example}
\newtheorem{remark}{Remark}
\begin{document}

%%%%%%%%%%%%%%%%%%%%%%%%%%%%%%%%%%%%%%%%%%%%%%%%%%%%%%%%%%%%%%%%%%%%%%%%%%%%%%%%%%%%%%%%%%%%%%%%%%%%%%%%%%%%%%%%%%%%%%%%%%%%
%%%%%%%%%%%%%%%%%%%%%%%%%%%%%%%%%%%%%%%%%%%%%%%%%%%%%%%%%%%%%%%%%%%%%%%%%%%%%%%%%%%%%%%%%%%%%%%%%%%%%%%%%%%%%%%%%%%%%%%%%%%%

\renewcommand{\baselinestretch}{1.6}

%\markright{ \hbox{\footnotesize\rm Statistica Sinica
%%{\footnotesize\bf 24} (201?), 000-000
%}\hfill\\[-13pt]
%\hbox{\footnotesize\rm
%%\href{http://dx.doi.org/10.5705/ss.20??.???}{doi:http://dx.doi.org/10.5705/ss.20??.???}
%}\hfill }

%\markboth{\hfill{\footnotesize\rm Ting Tao Shaohua Pan AND Shujun Bi} \hfill}
% {\hfill{\footnotesize\rm CALIBRATED ZERO-NORM REGULARIZED ESTIMATOR} \hfill}

%\renewcommand{\thefootnote}{}
%$\ $\par

%%%%%%%%%%%%%%%%%%%%%%%%%%%%%%%%%%%%%%%%%%%%%%%%%%%%%%%%%%%%%%%%%%%%%%%%%%%%%%%%%%%%%%%%%%%%%%%%%%%%%%%%%%%%%%%%%%%%%%%%%%%%

\fontsize{12}{14pt plus.8pt minus .6pt}\selectfont \vspace{0.8pc}
\centerline{\large\bf Calibrated zero-norm regularized LS estimator for}
\vspace{2pt} \centerline{\large\bf high-dimensional error-in-variables regression}
\vspace{.4cm} \centerline{Ting Tao,\ \ Shaohua Pan\ \ and\ \ Shujun Bi}
\vspace{.4cm} \centerline{August 21, 2019}
\vspace{.55cm} \fontsize{9}{11.5pt plus.8pt minus.6pt}\selectfont

%%%%%%%%%%%%%%%%%%%%%%%%%%%%%%%%%%%%%%%%%%%%%%%%%%%%%%%%%%%%%%%%%%%%%%%%%%%%%%%%%%%%%%%%%%%%%%%%%%%%%%%%%%%%%%%%%%%%%%%%%%%%

\begin{abstract}
  This paper is concerned with high-dimensional error-in-variables regression
  that aims at identifying a small number of important interpretable factors
  for corrupted data from the applications where measurement errors or missing data
  can not be ignored. Motivated by CoCoLasso due to \cite{Datta16}
  and the advantage of the zero-norm regularized LS estimator over Lasso for clean data,
  we propose a calibrated zero-norm regularized LS (CaZnRLS) estimator by
  constructing a calibrated least squares loss with a positive definite
  projection of an unbiased surrogate for the covariance matrix of covariates,
  and use the multi-stage convex relaxation approach to compute this estimator.
  Under a restricted strong convexity on the true covariate matrix,
  we derive the $\ell_2$-error bound of every iterate and establish
  the decreasing of the error bound sequence and the sign consistency
  of the iterates after finite steps. The statistical guarantees are also
  provided for the CaZnRLS estimator under two types of measurement errors.
  Numerical comparisons with CoCoLasso and NCL (the nonconvex Lasso
  of \cite{Loh11}) show that CaZnRLS has better relative RMSE
  as well as comparable even more correctly identified predictors.
\end{abstract}
\vspace{9pt}
\par
\par

%\lhead[\footnotesize\thepage\fancyplain{}\leftmark]{}\rhead[]{\fancyplain{}\rightmark\footnotesize\thepage}
%Put this line in Page 2

\def\thefigure{\arabic{figure}}
\def\thetable{\arabic{table}}

\renewcommand{\theequation}{\thesection.\arabic{equation}}

\fontsize{12}{14pt plus.8pt minus .6pt}\selectfont

\setcounter{section}{1} %***
\setcounter{equation}{0} %-1

\noindent {\bf 1. Introduction}
 \par
 Over the past decade or so, high-dimensional regression is found to
 have wide applications in various fields such as genomics, finance,
 image processing, climate science, sensor network, and so on.
 The canonical high-dimensional linear regression model assumes that
 the number of available predictors $p$ is larger than the sample
 size $n$, although the number of true relevant predictors $s$
 is much less than $p$. This model can be expressed as
 \begin{equation}\label{observation}
   y=X\beta^*+\varepsilon
 \end{equation}
 where $y=(y_1,\ldots,y_n)^{\mathbb{T}}$ is the vector of responses,
 $X=(x_{ij})$ is the $n\times p$ matrix of covariates, $\beta^*\in\mathbb{R}^p$
 is a sparse coefficient vector with $s$ nonzero entries, and
 $\varepsilon=(\varepsilon_1,\ldots,\varepsilon_n)^{\mathbb{T}}$
 is the noise vector. Unless otherwise states, we assume that
 all covariates are centered so that the intercept term is not included
 in \eqref{observation} and the matrix $X$ of covariates has normalized columns.

 The current popular high-dimensional regression methods include
 convex type estimators such as Lasso in \cite{LASSO94}, adaptive
 Lasso in \cite{Zou06}, elastic net in \cite{elasticnet05} and Dantzig selector
 in \cite{Dantzig01}; and nonconvex type estimators such as SCAD in \cite{FanLi01}
 and MCP in \cite{Zhang10}. The reader may refer to the article of \cite{FanLi10}
 and the monograph of \cite{Bhlmann11} for an excellent overview of these methods.
 They are to some extent imitating the performance of the zero-norm penalized LS estimator
 \begin{equation}\label{znorm}
   \beta^{\rm zn}\in\mathop{\arg\min}_{\|\beta\|_{\infty}\le R}
   \,\left\{\frac{1}{2n\lambda}\|y-X\beta\|^2+\|\beta\|_0\right\}
 \end{equation}
 where the ball constraint $\|\beta\|_{\infty}\le R$ for some $R>0$
 ensures the well-definedness of $\beta^{\rm zn}$, and $\lambda>0$ is
 the regularization parameter. Recently, by developing a global exact
 penalty for the equivalent mathematical program with equilibrium constraints
 (MPEC), \cite{BiPan17} showed that a global optimal solution of \eqref{znorm}
 can be obtained from the solution of a global exact penalization problem,
 and the popular SCAD estimator is the product yielded by eliminating the dual
 part of a global exact penalization problem. By solving the global exact
 penalization problem in an alternating way, they proposed a multi-stage
 convex relaxation approach (GEP-MSCRA), which can be regarded as
 an adaptive Lasso embedded with the dual information.
 We notice that for the clean design matrix $X$, the zero-norm regularized
 LS estimator computed with GEP-MSCRA has a remarkable advantage over Lasso
 in reducing the prediction error and capturing sparsity.

 In many applications, we often face corrupted data due to inaccurate
 observations for covariates or missing values. Common examples include
 sensor network data (see \cite{Slijepcevic02}), high-throughout sequencing
 (see \cite{Benjamini12}), and gene expression data (see \cite{Purdom05}).
 In this setting, if the high-dimensional regression method for clean data
 is naively applied to the corrupted data, one will obtain misleading inference
 results, see \cite{Rosenbaum10}. Then, it is natural to ask how to modify
  the zero-norm regularized LS estimator so that it can still display
  its strong points for the corrupted data. Motivated by CoCoLasso
  in \cite{Datta16}, we shall propose a calibrated
  zero-norm regularized LS estimator. For the convenience of discussion,
  we assume that a corrupted covariate matrix $Z=(z_{ij})_{n\times p}$
  instead of the true covariate matrix $X$ is observed. As mentioned in
  \cite{Loh14} and \cite{Datta16}, depending on the specific context, there are various
  ways to model the measurement errors. For example, in the additive noise setting,
  $Z\!=X+A$ where $A=(a_{ij})_{n\times p}$ is the additive noise matrix;
  in the multiplicative errors setup, $Z=X\circ M$ where $M=(m_{ij})_{n\times p}$
  is the matrix of multiplicative errors and ``$\circ$'' denotes the elementwise
  multiplication operator; and missing values can be
  viewed as a special case of multiplicative errors.

  The loss term $\frac{1}{2n}\|y-X\beta\|^2$ in the clean setting can be rewritten as
  \begin{equation}\label{hSigma}
    \frac{1}{2}\beta^{\mathbb{T}}\Sigma\beta-\xi^{\mathbb{T}}\beta+\frac{1}{2n}\|y\|^2
    \ \ {\rm with}\ \Sigma:=\frac{1}{n}X^{\mathbb{T}}X\ {\rm and}\ \xi:=\frac{1}{n}X^{\mathbb{T}}y.
  \end{equation}
  By recalling that the covariates are centered, it is easy to check that
  $(\Sigma,\xi)$ is an unbiased estimator of $(\Sigma_x,\Sigma_x\beta^*)$
  where $\Sigma_x$ denotes the covariance matrix of the covariates.
  With the corrupted $Z$ and $y$, \cite{Loh11} constructed an unbiased
  surrogate $(\widehat{\Sigma},\widehat{\xi})$ of $(\Sigma,\xi)$,
  and obtained an estimation for the true $\beta^*$ via
  the following optimization model
  \begin{equation}\label{CLASSO}
   \widehat{\beta}\in\mathop{\arg\min}_{\|\beta\|_1\le R_0}
   \left\{\frac{1}{2}\beta^{\mathbb{T}}\widehat{\Sigma}\beta
   -\widehat{\xi}^{\mathbb{T}}\beta+\lambda_n\|\beta\|_1 \right\}.
  \end{equation}
  Notice that the unbiased surrogate $\widehat{\Sigma}$ constructed with $Z$
  may not be positive semidefinite; for example, when $x_{ij}$ is corrupted
  by the independent additive errors $a_{ij}$ with mean $0$ and variance $\tau^2$,
  the matrix $\widehat{\Sigma}=\frac{1}{n}Z^{\mathbb{T}}Z-\tau^2 I$ is
  an unbiased surrogate for $\Sigma$ which has a negative eigenvalue due to $n<p$.
  So, the objective function of \eqref{CLASSO} may be nonconvex and lower unbounded.
  Loh and Wainwright introduced the constraint $\|\beta\|_1\le R_0$ in
  the model \eqref{CLASSO} to guarantee that it has an optimal solution.
  Through some careful analysis, they showed that if $R_0$ is properly chosen,
  a projected gradient descent algorithm will converge in polynomial time
  to a small neighborhood of the set of all global minimizers.
  However, as remarked in \cite{Datta16}, the practical performance of
  the nonconvex Lasso model \eqref{CLASSO} depends greatly on the choice of $R_0$.
  Similar shortcoming also appears in the procedure of \cite{Chen13}.

  To overcome the shortcoming of \eqref{CLASSO} and enjoy
  the convex formulation of Lasso, \cite{Datta16} recently
  proposed a convex conditioned Lasso (CoCoLasso).
  Let $W\succeq\widehat{\epsilon}I$ mean that $W-\widehat{\epsilon}I$ is
  positive semidefinite (PSD) and $\|Z\|_{\rm max}=\max_{i,j}|z_{ij}|$
  denotes the elementwise maximum norm of a matrix $Z$.
  They first solved the following PSD
  optimization problem
   \begin{equation}\label{Convex1}
   \overline{\Sigma}\in\mathop{\arg\min}_{W\succeq\widehat{\epsilon}I}
   \|W-\widehat{\Sigma}\|_{\rm max}\ \ {\rm for\ some}\ \widehat{\epsilon}>0
  \end{equation}
  to obtain a nearest PD approximation to
  the unbiased surrogate $\widehat{\Sigma}$ of $\Sigma$ constructed as in \cite{Loh11}
  with $Z$, and then define
  \begin{equation}\label{CoCoLasso}
   \overline{\beta}=\mathop{\arg\min}_{\beta\in\mathbb{R}^p}
   \left\{\frac{1}{2n}\|\overline{y}-\overline{Z}\beta\|^2+\lambda\|\beta\|_1\right\}
  \end{equation}
  with the Cholesky factor
  $\overline{Z}/\sqrt{n}$ of $\overline{\Sigma}$ and the vector $\overline{y}$
  satisfying $\overline{Z}^{\mathbb{T}}\overline{y}=Z^{\mathbb{T}}y$.

  The elementwise maximum norm in the model \eqref{Convex1} plays
  a twofold role: measuring the approximation of $\overline{\Sigma}$
  to $\Sigma$ and removing a certain noise involved in $\widehat{\Sigma}$.
  Compared with other elementwise norms such as $\ell_1$-norm and Frobenius norm,
  the maximum norm indeed yields an approximation whose entries are
  closer to those of $\widehat{\Sigma}$. However, the computation of
  $\overline{\Sigma}$ is expensive since the model \eqref{Convex1} is
  a convex program of $p^2$ variables which involves two nonsmooth terms:
  the objective function $\|W-\widehat{\Sigma}\|_{\rm max}$ and
  the PSD constraint. Figure \ref{figZou} below indicates that
  when using the alternating direction method of multipliers (ADMM)
  described in Appendix A of \cite{Datta16} to solve \eqref{Convex1}
  with $\widehat{\Sigma}$ from the data in Subsection 5.1.1,
  the computing time increases quickly with the increase of $p$ or
  the accuracy improvement of solution. Consider that the problem
  \eqref{Convex1} aims at seeking an approximation to the covariance matrix
  $\Sigma_x$ instead of the noisy unbiased surrogate $\widehat{\Sigma}$.
  It is reasonable to seek an approximation which has a little worse
  approximate accuracy but can be cheaply achieved, and then employ a more effective
  high-dimensional regression method than Lasso to define an estimator.
  When the elementwise maximum norm in \eqref{Convex1} is replaced by
  the Frobenius norm, its solution is exactly the projection of
  $\widehat{\Sigma}-\widehat{\epsilon}I$ onto the PSD cone and
  can be obtained from one eigenvalue decomposition for $\widehat{\Sigma}$.
  Also, when $\widehat{\Sigma}=\frac{1}{n}Z^{\mathbb{T}}Z-\tau^2 I$, this solution
  well matches the structure of $\widehat{\Sigma}$.
  Motivated by this, we replace the objective function of \eqref{Convex1}
  with the Frobenius norm of $W-\widehat{\Sigma}$ to obtain an approximation
  $\widetilde{\Sigma}$, and with its eigenvalue decomposition
  define a zero-norm regularized LS estimator.

  \begin{figure}[t]
  \centerline{\epsfig{file=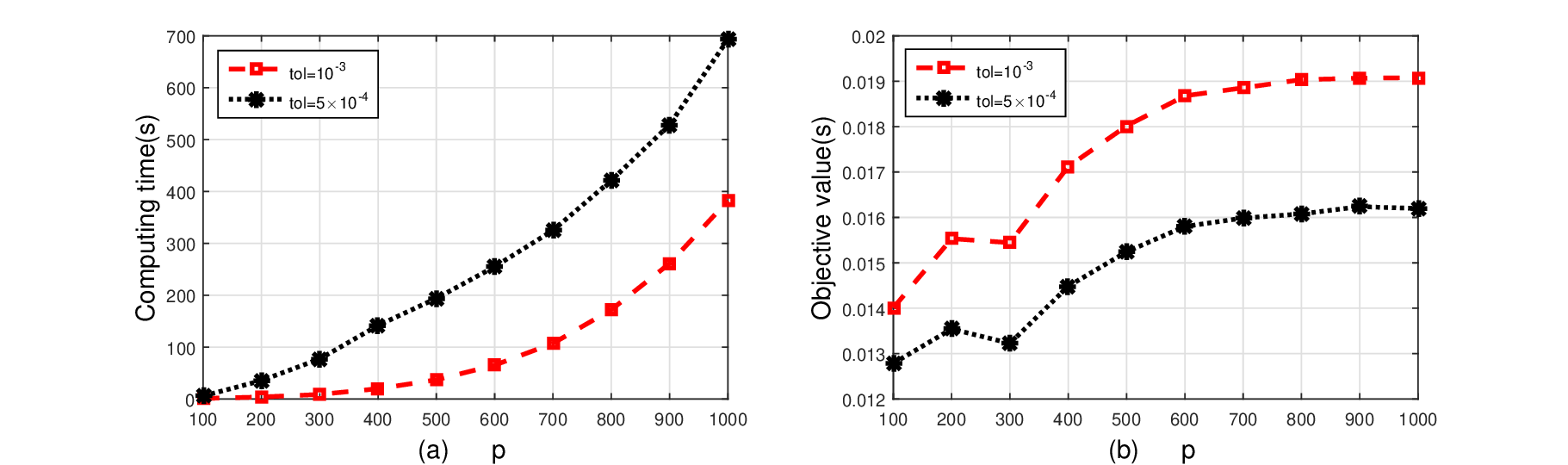,width=4.5in}}
  \par \caption{\footnotesize The computing time and objective value of Algorithm 1 due to \cite{Datta16}
  with the stopping condition
  $\max(\|A_{i+1}\!-A_i\|_F,\|B_{i+1}\!-B_i\|,\|\Lambda_{i+1}\!-\Lambda_i\|)\le{\rm tol}$}
  \label{figZou}
 \end{figure}

 We also notice that a Dantzig selector type estimator and its improved version
  were proposed in \citep{Rosenbaum10,Rosenbaum13} and \cite{Belloni17}, respectively,
  for additive measurement error models. Since these estimators are defined via
  an optimization problem with an D.C. (difference of convexity) constraint,
  it is difficult to obtain these estimators in practice. To overcome
  the difficulty caused by the D.C. constraint, they recently relaxed
  the nonconvex constraint set to a convex set and proposed two conic
  programming based on estimators for the same model setup (see \cite{Belloni16}),
  which can be viewed as a relaxed version of the Dantzig selector
  for the clean data. In addition, \cite{Stadler12} derived an algorithm
  for sparse linear regression with missing data based on a sparse inverse
  covariance matrix estimation. In the same spirit of \cite{Loh11} and \cite{Datta16},
  we propose the CaZnRLS estimator that can handle simultaneously
  additive errors, multiplicative errors and missing data case.
  Although the CaZnRLS estimator is defined by a nonconvex
  optimization problem, the GEP-MSCRA in \cite{BiPan17}
  (see Section \ref{sec3}) provides an efficient solver for it,
  which consists of solving a sequence of weighted $\ell_1$-regularized LS problems.
  As shown by the simulation study in Section \ref{sec5}, the estimator still
  displays its merits in reducing prediction error and capturing sparsity for
  the contaminated data as it does for the clean data.

  The rest of this paper is organized as follows. In Section \ref{sec2},
  we define the calibrated zero-norm regularized LS estimator and provide
  a primal-dual view on this estimator. Section \ref{sec3} describes
  the GEP-MSCRA solver for computing the CaZnRLS estimator. In Section \ref{sec4},
  under a restricted eigenvalue assumption on the matrix $\Sigma$,
  we provide the deterministic theoretical guarantees including
  the $\ell_2$-error bound for every iterate, the decreasing of the error
  bound sequence, and the sign consistency of the iterates after finite steps;
  and the statistical guarantees for the computed estimator under two types
  of measurement errors. In Section \ref{sec5}, we compare the performance
  of CaZnRLS with that of CoCoLasso and NCL.

  To close this section, we introduce some necessary notations.
  Let $\mathbb{S}^p$ be the space consisting of all $p\times p$
  real symmetric matrices, equipped with the trace inner product
  $\langle W,Y\rangle={\rm trace}(W^{\mathbb{T}}Y)$ and its induced
  Frobenius norm $\|\cdot\|_F$, and let $\mathbb{S}_{+}^p$ be the cone
  consisting of all PSD matrices in $\mathbb{S}^p$. For any symmetric
  matrix $W$, let $\lambda_{\rm min}(W)$ and $\lambda_{\rm max}(W)$
  denote the smallest and largest eigenvalues of $W$.
  For any vector $z$, $\|z\|_\infty$ denotes the infinity norm of $z$.
  Let $I$ and $e$ denote an identity matrix and a vector of all ones, respectively,
  whose dimensions are known from the context. For a closed set $\Omega$,
  $\delta_{\Omega}(\cdot)$
  denotes the indicator function on $\Omega$, i.e., $\delta_{\Omega}(x)=0$
  if $x\in \Omega$ and otherwise $\delta_{\Omega}(x)=+\infty$; and when
  $\Omega$ is convex, $\Pi_{\Omega}(\cdot)$ denotes the projection
  operator onto $\Omega$. For an index set $\Lambda\subseteq\{1,\ldots,p\}$,
  write $\Lambda^{c}:=\{1,\ldots,p\}\backslash \Lambda$ and denote by
  $\mathbb{I}_{\Lambda}(\cdot)$ the characterization function on $\Lambda$,
  and $Y_{\Lambda}$ by the submatrix of $Y$ consisting of the column $Y_j$
  with $j\in\Lambda$. For any nonnegative real number $a$, $\lfloor a\rfloor$
  and $\lceil a\rceil $ denote the largest integer less than $a$ and
  the smallest integer greater than $a$, respectively.

%------------------------------------------------------------------------------
%\setcounter{section}{2} %***
\setcounter{equation}{0} %-1
%\noindent {\bf 2. The Second Section}
 \section{Calibrated zero-norm regularized LS estimator}\label{sec2}

  When the data is corrupted by measurement errors, the observed matrix
  $Z$ of predictors is a function of the true covariate matrix $X$ and
  random errors. In this case, one may construct an unbiased surrogate
  $(\widehat{\Sigma},\widehat{\xi})$ for the pair $(\Sigma,\xi)$ with
  $Z$ and $y$ as in \cite{Loh11}. For the specific form of
  $(\widehat{\Sigma},\widehat{\xi})$ under various types of measurement errors,
  one may refer to Section 2 in \cite{Loh11} or see Appendix B.
  Now assume that an unbiased surrogate $(\widehat{\Sigma},\widehat{\xi})$ is available.
  Let $\widehat{\Sigma}$ have the eigenvalue decomposition
  as $\widehat{\Sigma}=P{\rm Diag}(\theta_1,\ldots,\theta_p)P^{\mathbb{T}}$
  where $P$ is a $p\times p$ orthonormal matrix and $\theta_1\ge\theta_2\ge\cdots\ge\theta_p$
  are the eigenvalues of $\widehat{\Sigma}$.

  Since it is time-consuming to compute a solution of \eqref{Convex1}
  when $p$ is large, we replace the elementwise maximum norm
  in \eqref{Convex1} with the Frobenius norm and achieve a nearest PD
  approximation to $\widehat{\Sigma}$ via the following model
  \begin{equation}\label{corr-prob}
   \widetilde{\Sigma}=\mathop{\arg\min}_{W\succeq\widehat{\epsilon}I}\|W-\widehat{\Sigma}\|_F.
  \end{equation}
  Note that \eqref{corr-prob} has the same solution set as
  $\min_{W\succeq\widehat{\epsilon} I}\|W-\widehat{\Sigma}\|_F^2$ does. So,
  \begin{equation}\label{Decom}
  \widetilde{\Sigma}=\widehat{\epsilon}I+\Pi_{\mathbb{S}_{+}^p}(\widehat{\Sigma}\!-\!\widehat{\epsilon}I)
   =P{\rm Diag}\big(\max(\theta_1,\widehat{\epsilon}),\ldots,\max(\theta_p,\widehat{\epsilon})\big)P^{\mathbb{T}}.
  \end{equation}
  Clearly, when $\widehat{\Sigma}=\frac{1}{n}Z^{\mathbb{T}}Z-\tau^2I$,
  a composite of a low-rank matrix and an identity matrix,
  the solution $\widetilde{\Sigma}$ keeps this structure.
  Also, one eigenvalue decomposition of $\frac{1}{n}Z^{\mathbb{T}}Z$
  is enough to formulate the solution $\widetilde{\Sigma}$. Indeed, let
  \begin{equation}\label{WSigma-y}
    \left\{\begin{array}{ll}
    \widetilde{Z}:=\sqrt{n}P{\rm Diag}\Big(\sqrt{\max(\theta_1,\widehat{\epsilon})},\ldots,\sqrt{\max(\theta_p,\widehat{\epsilon})}\Big)P^{\mathbb{T}},\\
    \widetilde{y}:=\sqrt{n}P{\rm Diag}\Big(\frac{1}{\sqrt{\max(\theta_1,\widehat{\epsilon})}},\ldots,\frac{1}{\sqrt{\max(\theta_p,\widehat{\epsilon})}}\Big)P^{\mathbb{T}}\widehat{\xi}.
    \end{array}\right.
  \end{equation}
  By invoking \eqref{Decom}, one may check that
  $\widetilde{\Sigma}=\frac{1}{n}\widetilde{Z}^{\mathbb{T}}\widetilde{Z}$
  and $\widehat{\xi}=\frac{1}{n}\widetilde{Z}^{\mathbb{T}}\widetilde{y}$.

  Although the computation of $\widetilde{\Sigma}$ becomes much
  cheaper than that of $\overline{\Sigma}$, it is a worse approximation
  to $\widehat{\Sigma}$ than the latter since the minimization
  of the elementwise maximum norm tends to give smaller entries.
  This requires us to define an estimator by more effective high-dimensional
  regression methods than Lasso. A natural candidate is the nonconvex type
  estimator such as SCAD and MCP since they can remove the bias of Lasso.
  Note that SCAD and MCP functions are actually imitating the performance
  of zero-norm. We define the zero-norm regularized LS estimator
  \begin{equation}\label{Beta-sparse}
   \widetilde{\beta}\in\mathop{\arg\min}_{\beta\in\mathbb{R}^p}
   \Big\{\frac{1}{2n\lambda}\|\widetilde{Z}\beta-\widetilde{y}\|^2+\|\beta\|_{0}\Big\}.
  \end{equation}
  Taking into account that $(\widetilde{Z},\widetilde{y})$ is a calibrated
  pair of $(\widehat{\Sigma},\widehat{\xi})$, we call \eqref{Beta-sparse}
  a calibrated version of the zero-norm regularized LS estimator
  defined with the corrupted observation $Z$ as in \eqref{znorm},
  except that the ball constraint is now removed due to the coerciveness
  of the strong convex $\|\widetilde{Z}\beta-\widetilde{y}\|^2$.
  Compared with the SCAD estimator, the solution of \eqref{Beta-sparse}
  seems to be much more difficult since the problem \eqref{Beta-sparse} is
  even discontinuous due to the combinatorial property of zero-norm.
  However, as demonstrated later, the SCAD estimator is actually
  equivalent to the zero-norm regularized LS.

  Next we shall provide a primal-dual look into the estimator $\widetilde{\beta}$.
  Define
  \begin{equation}\label{phi}
    \phi(t):=\frac{a\!-\!1}{a\!+\!1}t^2+\frac{2}{a\!+\!1}t\ \ (a>1)\quad\ {\rm for}\ \ t\in\mathbb{R}.
  \end{equation}
  With this function, it is immediate to check that for any $\beta\in\mathbb{R}^p$,
  \begin{equation*}\label{vc-zn}
   \|\beta\|_0=\min_{w\in\mathbb{R}^p}\!\bigg\{\sum_{i=1}^p\phi(w_i)\!: \langle e-w,|\beta|\rangle=0,
   \ 0\le w\le e\bigg\}.
 \end{equation*}
 This shows that the zero-norm is essentially an optimal value function
 of a parameterized MPEC, since $\langle e-w,|\beta|\rangle=0$ and $e-w\ge 0$
 constitute an equilibrium constraint. Thus, \eqref{Beta-sparse} is
 equivalent to the following MPEC
 \begin{equation}\label{vc-sparse}
   \min_{\beta,w\in\mathbb{R}^p}\!\bigg\{\frac{1}{2n\lambda}\|\widetilde{Z}\beta-\widetilde{y}\|^2
   +\sum_{i=1}^p\phi(w_i)\!:\, \langle e-w,|\beta|\rangle=0,\,0\le w\le e\bigg\},
 \end{equation}
 in the sense that if $\widetilde{\beta}^{\natural}$ is a global optimal solution of \eqref{Beta-sparse},
 then $(\widetilde{\beta}^{\natural},{\rm sign}(|\widetilde{\beta}^{\natural}|))$
 is globally optimal to \eqref{vc-sparse};
 and conversely, if $(\widetilde{\beta}^{\natural},\widetilde{w}^{\natural})$
 is a global optimal solution of \eqref{vc-sparse}, then $\widetilde{\beta}^{\natural}$
 is globally optimal to \eqref{Beta-sparse} with
 $\|\widetilde{\beta}^{\natural}\|_{0}=\sum_{i=1}^p\phi(\widetilde{\beta}^{\natural}_i)$.

 The MPEC form \eqref{vc-sparse} shows that the difficulty
 to compute the estimator $\widetilde{\beta}$ arises from
 the constraint $\langle e-\!w,|\beta|\rangle=0$, which brings
 the bothersome nonconvexity. Since it is much harder to handle
 nonconvex constraints than to handle nonconvex objective,
 we consider its penalized version
 \begin{equation}\label{pvc-sparse}
   \min_{\beta\in\mathbb{R}^p,w\in[0,e]}\!\bigg\{\frac{1}{2n\lambda}\|\widetilde{Z}\beta-\widetilde{y}\|^2
   +\sum_{i=1}^p\phi(w_i) +\rho\langle e-w,|\beta|\rangle\bigg\}
 \end{equation}
 where $\rho>0$ is the penalty parameter. By the coerciveness of
 the function $\beta\mapsto\|\widetilde{Z}\beta-\widetilde{y}\|^2$,
 there exists a constant $\widehat{R}>0$ such that \eqref{vc-sparse} and \eqref{pvc-sparse}
 are equivalent to their respective version in which the variable $\beta$
 is required to lie in the set $\{\beta\in\mathbb{R}^p\ |\ \|\beta\|_{\infty}\le\widehat{R}\}$.
 Thus, by invoking Theorem 2.1 of \cite{BiPan17}, we have the following result.
%-------------------------------------------------------------------Theorem 2.1
 \begin{theorem}\label{exact-penalty}
  Let $L_{\!f}$ be the Lipschitz constant of
  $f(\beta):=\frac{1}{2n}\|\widetilde{Z}\beta-\widetilde{y}\|^2$
  on the ball $\{\beta\in\mathbb{R}^p\!:\|\beta\|_{\infty}\le\widehat{R}\}$. Then,
  for every $\rho\ge\overline{\rho}:=\frac{4aL_{\!f}}{(a+1)\lambda}$,
  the global optimal solution set of \eqref{pvc-sparse} associated to
  $\rho$ coincides with that of \eqref{vc-sparse}.
 \end{theorem}

 Theorem \ref{exact-penalty} shows that the problem \eqref{pvc-sparse}
 is a global exact penalty of \eqref{vc-sparse} in the sense that it has
 the same global optimal solution set as \eqref{vc-sparse} does once $\rho$
 is greater than a threshold. Consequently, the estimator $\widetilde{\beta}$
 can be achieved by solving the following exact penalty problem with $\rho>\overline{\rho}$:
 \begin{equation}\label{pvc-sparse1}
 \widetilde{\beta}\in\mathop{\arg\min}_{\beta\in\mathbb{R}^p,w\in[0,e]}
  \!\bigg\{\frac{1}{2n}\|\widetilde{Z}\beta-\widetilde{y}\|^2
   +\sum_{i=1}^p\lambda\Big[\phi(w_i) +\rho(1-w_i)|\beta_i|\Big]\bigg\}.
 \end{equation}
 Compared with \eqref{Beta-sparse}, the problem \eqref{pvc-sparse1}
 involves an additional variable $w\in\mathbb{R}^p$ which provides a part of
 dual information on \eqref{Beta-sparse}. Hence, \eqref{pvc-sparse1}
 can be viewed as a primal-dual equivalent form of \eqref{Beta-sparse}.
 This form does not involve the combinatorial difficulty, and its nonconvexity
 is just due to the coupled term $\langle w,|\beta|\rangle$ which is clearly
 much easier to cope with. In particular, the SCAD function in \cite{FanLi01} is precisely
 the optimal value of the inner minimization in \eqref{pvc-sparse1} w.r.t. $w$.
 To see this, we define
 \begin{equation}\label{phi-psi}
  \psi(t):=\!\left\{\!\begin{array}{cl}
                  \phi(t) &\textrm{if}\ t\in [0,1],\\
                   +\infty & \textrm{otherwise}.
                 \end{array}\right.
 \end{equation}
 Recalling the conjugate
 $\psi^*(\omega)=\sup_{t\in\mathbb{R}}\{t\omega-\psi(t)\}$ of $\psi$ by \cite{Roc70},
 we can compactly write the inner minimization in \eqref{pvc-sparse1} w.r.t. $w$ as
 \begin{equation*}
  \min_{w\in\mathbb{R}^p}\!\Big\{{\textstyle\sum_{i=1}^p}\lambda\big[\psi(w_i) +\rho(1\!-\!w_i)|\beta_i|\big]\Big\}
  ={\textstyle\sum_{i=1}^p}\lambda\big[\rho|\beta_i|-\psi^*\big(\rho|\beta_i|\big)\big].
 \end{equation*}
 After an elementary calculation, the conjugate $\psi^*$ of $\psi$ has the form of
 \begin{equation*}
    \psi^*(\omega)=\left\{\!\begin{array}{cl}
                      0 & \textrm{if}\ \omega\leq \frac{2}{a+1},\\
                       \frac{((a+1)\omega-2)^2}{4(a^2-1)}& \textrm{if}\ \frac{2}{a+1}<\omega\le \frac{2a}{a+1},\\
                      \omega-1 & \textrm{if}\ \omega>\frac{2a}{a+1}.
                \end{array}\right.
 \end{equation*}
  By comparing with the expression of the SCAD function $p_{\gamma}(t)$, the function
 $\lambda[\rho|t|-\psi^*(\rho|t|)]$ with $\lambda=\frac{(a+1)\gamma^2}{2}$
 and $\rho=\frac{2}{(a+1)\gamma}$ reduces to $p_{\gamma}(t)$. Thus,
 \begin{equation}\label{SCAD-prob}
   \widetilde{\beta}\in\mathop{\arg\min}_{\beta\in\mathbb{R}^p}
   \!\Big\{\frac{1}{2n}\|\widetilde{Z}\beta-\widetilde{y}\|^2
   +{\textstyle\sum_{i=1}^p}p_{\gamma}(|\beta_i|)\Big\}.
 \end{equation}

%%%%%%%%%%%%%%%%%%%%%%%%%%%%%%%%%%%%%%%%%%%%%%%%%%%%%%%%%%%%%%%%%%%%%%%%%%%%%%%%%%%%%%%%%%%%%%%%%%%%%%%%%%%%%%%%%%%%%%%%%%%%

%\setcounter{section}{3} %***
\setcounter{equation}{0} %-1
%\noindent {\bf 3. The Third Section}
\section{GEP-MSCRA for computing the estimator $\widetilde{\beta}$}\label{sec3}

 From the last section, to compute the estimator $\widetilde{\beta}$,
 one only needs to solve a single penalty problem \eqref{pvc-sparse1}
 which is much easier than \eqref{Beta-sparse} since
 its nonconvexity is from the coupled term $\langle w,|\beta|\rangle$.
 The GEP-MSCRA proposed in \cite{BiPan17} makes good use of the coupled
 structure and solves the problem \eqref{pvc-sparse1} in an alternating way.
 Since the threshold $\overline{\rho}$ is unknown though one may
 obtain an upper estimation for it, a varying $\rho$ is introduced in GEP-MSCRA.
 The iterations of GEP-MSCRA are described below.
%-------------------------------------------------------------------------------------
 \begin{algorithm}[!h]
 \caption{\label{Alg}{\bf\ \ GEP-MSCRA for computing $\widetilde{\beta}$}}
 \textbf{Initialization:} Choose $\lambda>0,\rho_0=1$ and an initial $w^0\in[0,\frac{1}{2}e]$.
                          Set $k:=1$.\\
 \textbf{while} the stopping conditions are not satisfied \textbf{do}
 \begin{enumerate}
  \item  Compute the following minimization problem
         \begin{equation}\label{expm-subx}
          \beta^k=\mathop{\arg\min}_{\beta\in\mathbb{R}^p}
           \bigg\{\frac{1}{2n}\big\|\widetilde{Z}\beta-\widetilde{y}\big\|^2
                 +\lambda\sum_{i=1}^p(1\!-\!w_i^{k-1})|\beta_{i}|\bigg\}.
         \end{equation}

  \item When $k=1$, select a suitable $\rho_1\ge\rho_0$ in terms of $\|\beta^1\|_\infty$.
        Otherwise, select $\rho_k$ such that $\rho_k\ge\rho_{k-1}$
        for $k\le 3$; and $\rho_k=\rho_{k-1}$ for $k>3$.

  \item Seek the unique optimal solution $w_i^k\ (i=1,\ldots,p)$ of the problem
        \begin{equation}\label{expm-subw}
         w_i^k=\mathop{\rm arg\min}_{0\le w_i\le 1}
              \left\{\phi(w_i)-\rho_k w_i|\beta^k_{i}|\right\}.
        \end{equation}

  \item Let $k\leftarrow k+1$, and then go to Step 1.
  \end{enumerate}
  \textbf{end while}
\end{algorithm}
\begin{remark}\label{remark-Alg}
 {\bf(a)} Since $\phi$ is strongly convex, the problem \eqref{expm-subw}
 has a unique optimal solution. By the expression of $\phi$,
 it is immediate to obtain
 \begin{equation}\label{wik}
   w_i^k=\min\Big[1,\max\Big(\frac{(a+1)\rho_k|\beta^k_{i}|-2}{2(a-1)},0\Big)\Big]
   \ \ {\rm for}\ \ i=1,2,\ldots,p.
 \end{equation}
 Thus, the main computation work of GEP-MSCRA in each step
 is solving a weighted $\ell_1$-regularized LS. In this sense,
 GEP-MSCRA is analogous to the local linear approximation
 algorithm \cite{Zou08} applied to the problem \eqref{SCAD-prob} except
 the start-up and the weights, where the start-up of the former depends
 explicitly on the dual variable $w^0$, while that of the latter
 depends implicitly on a good estimator $\beta^0$. So,
 when computing CaZnRLS with GEP-MSCRA, one actually obtains
 an adaptive Lasso estimator. The initial $w^0$ may be
 an arbitrary vector from the box set $[0,\frac{1}{2}e]$. Here,
 we restrict $w^0$ to the box set $[0,\frac{1}{2}e]$, rather than
 the feasible set $[0,e]$ of $w$ in \eqref{pvc-sparse1}, so as to achieve
 a better initial estimator $\beta^{1}$.

 \noindent
 {\bf(b)} Due to the combinatorial property of $\|\cdot\|_0$,
 it is almost impossible to get $\widetilde{\beta}$ exactly.
 The popular Lasso of \cite{LASSO94} or adaptive Lasso of \cite{Zou06},
 as a one-step or series of convex relaxation to \eqref{Beta-sparse},
 arises from the primal angle, while the series of weighted
 $\ell_1$-norm regularized LS problems in GEP-MSCRA arise from
 the primal-dual reformulation of \eqref{Beta-sparse}.

 \noindent
 {\bf(c)} From the formula \eqref{wik}, if $\rho_k|\beta^k_{i}|$ is larger,
 then $w_i^k$ has a value close to $1$, which means that in the $(k\!+\!1)$th iterate,
 a smaller weight $(1\!-\!w_i^{k})$ is imposed on the variable $\beta_i$,
 and consequently a conservative strategy is used for sparsity. Consider that
 for some difficult problems, the solution $\beta^1$ yielded by
 the $\ell_1$-regularized LS problem may not have a sharp gap
 between its nonzero and zero entries. Hence, in order to guarantee
 that the subsequent $\beta^k$ has a correct sparse support,
 we increase $\rho_k$ for $k\le 3$ appropriately, i.e., cut down
 the smaller nonzero entries conservatively, while for $k>3$ since
 $\beta^k$ generally has a big difference between its nonzero and zero entries,
 we keep $\rho_k$ unchanged so as to cut down the smaller nonzero entries quickly.
 \end{remark}

 In Appendix C, we provide the implementation details of GEP-MSCRA
 with the semismooth Newton augmented Lagrangian method (ALM) applied to
 the dual of \eqref{expm-subx}. As discussed in \cite{LiSunToh16},
 the semismooth Newton ALM fully exploits the second-order information of
 and the good structure of its dual and can yield a solution of high accuracy.

%------------------------------------------------------------------------
 \section{Theoretical guarantees for the GEP-MSCRA}\label{sec4}

 In this section we denote by $S^{*}$ the support of the true vector $\beta^{*}$, and define
 \[
  \mathcal {C}(S^{*}):=\bigcup_{S\supset S^{*},|S|\le 1.5s}\!\Big\{\beta\in \mathbb{R}^p\!:
  \|\beta_{S^c}\|_1\le 3\|\beta_{S}\|_1\Big\}.
 \]
 We say that $\Sigma$ satisfies the $\kappa$-restricted eigenvalue condition (REC) or
 $X$ satisfies the $\kappa$-restricted strong convexity on $\mathcal {C}(S^{*})$ if
 $\kappa>0$ is such that
 \[
  \beta^{\mathbb{T}}\Sigma\beta=\frac{1}{n}\|X\beta\|^2
   \ge \kappa\|\beta\|^2\quad{\rm for\ all}\ \beta\in\mathcal{C}(S^*).
 \]
 This REC is a little stronger than the one used in \cite{Negahban12}
 for the clean Lasso and in \cite{Datta16} for CoCoLasso since
 $\mathcal{C}(S^*)\supseteq\big\{\beta\in \mathbb{R}^p\!:
  \|\beta_{(S^*)^c}\|_1\le 3\|\beta_{S^*}\|_1\big\}$, and is different from
 the $(L,S^*,N)$-restricted eigenvalue condition introduced in \cite{VanGeer09}.
 We shall provide the deterministic theoretical
 guarantees for GEP-MSCRA under this REC with appropriate $\lambda,\rho_1$
 and $\rho_3$, which include the error bound of every iterate $\beta^k$
 to the true $\beta^{*}$, the decrease analysis of the error sequence,
 and the sign consistency analysis of $\beta^k$ after finite steps.
 All proofs of the main results are included in Appendix A.

%--------------------------------------------------------------------------
 \noindent {\bf 4.1. Error bound sequence and its decrease}\label{sec4.1}

  To achieve the error bound of the iterate $\beta^k$ to
  the true $\beta^*$, we write
  \begin{equation}\label{wtveps}
   D:=\widehat{\Sigma}-\Sigma \ \ {\rm and}\ \
   \widetilde{\varepsilon}:=\widehat{\xi}-\widetilde{\Sigma}\beta^{*}.
  \end{equation}
  The following theorem states a deterministic result for the error bound.
 %-------------------------------------------------------------------------Theorem3.1
 \begin{theorem}\label{errbound1}
  Suppose $\Sigma$ satisfies the $\kappa$-REC on $\mathcal{C}(S^{*})$
  with $\kappa>24s\|D\|_{\rm max}$. If $\lambda$ and $\rho_3$ are chosen with
  $\lambda\ge 8\|\widetilde{\varepsilon}\|_{\infty}$
  and $\rho_3\le\frac{2(\kappa-24s\|D\|_{\rm max})}{5\sqrt{2}\lambda}$,
  then
 \begin{equation}\label{err-ineq1}
  \|\beta^k-\beta^{*}\|
  \le\frac{5\sqrt{s}\,\lambda}{2(\kappa-24s\|D\|_{\rm max})}\quad\ \forall k\in\mathbb{N}.
 \end{equation}
 \end{theorem}

 The error bound in Theorem \ref{errbound1} has the same order,
 i.e. $O(\lambda\sqrt{s})$, as established for the clean Lasso in \cite{Negahban12}.
 From the proof of Theorem 1 in \cite{Datta16},
 $\|D\|_{\rm max}\le\frac{\kappa}{64s}$ holds with a high probability.
 This means that there is a high probability for the error bound of $\beta^k$
 not greater than $\frac{4\lambda\sqrt{s}}{\kappa}$, which is
 a little better than the bound $\frac{4\sqrt{2}\lambda\sqrt{s}}{\kappa}$
 in \cite{Datta16} although $\lambda$ is allowed to be greater than
 $8\|\widetilde{\varepsilon}\|_{\rm \infty}$ instead of $2\|\widetilde{\varepsilon}\|_\infty$
 as in \cite{Datta16}.

 \medskip

 Theorem \ref{errbound1} provides an error bound for every iterate,
 but it does not tell us if the error bound of the current $\beta^k$ is
 better than that of the previous $\beta^{k-1}$. To seek the answer,
 we study the decrease of the error bound sequence by bounding
 $(1\!-\!w^{k}_i)^2$ for $i\in S^*$. Write $F^0:=S^*$ and for $k\in\mathbb{N}$ define
 \begin{equation}\label{FLambdak}
  F^{k}:=\Big\{i\!: \big||\beta_i^k|-|\beta_i^*|\big|\ge(\rho_k)^{-1}\Big\}
  \ {\rm and}\  \Lambda^k:=\Big\{i\!:|\beta_i^*|\le \frac{4a}{(a\!+\!1)\rho_k}\Big\}.
 \end{equation}
 By Lemma \ref{indexk}, $(1\!-w^{k}_i)^2$ for $i\in S^*$ can be controlled by $\max(\mathbb{I}_{\Lambda^k}(i),\mathbb{I}_{F^k}(i))$.
 As a consequence, we have the following error bound result
 involving $\mathbb{I}_{\Lambda^k}(i)$.
%-------------------------------------------------------------------------------------
 \begin{theorem}\label{errbound2}
  Suppose $\Sigma$ satisfies the $\kappa$-REC on $\mathcal{C}(S^{*})$
  with $\kappa>24s\|D\|_{\rm max}$. If $\lambda$ and
  $\rho_3$ are chosen in the same way as in Theorem \ref{errbound1},
  then
  \begin{align*}
  \big\|\beta^k\!-\beta^{*}\big\|
  &\leq\frac{4+2\sqrt{2}}{\kappa-24s\|D\|_{\rm max}}\big\|\widetilde{\varepsilon}_{\!S^{*}}\big\|
    +\Big(\frac{1}{\sqrt{2}}\Big)^{k-1}\big\|\beta^{1}\!-\beta^{*}\big\|\\
    &\quad+\frac{2\lambda}{\kappa-24s\|D\|_{\rm max}}\sum_{j=1}^{k-1}\!\sqrt{{\textstyle\sum_{i\in S^{*}}}\mathbb{I}_{\Lambda^j}(i)}\Big(\frac{1}{\sqrt{2}}\Big)^{k-1-j}\quad\forall k\in\mathbb{N}.
  \end{align*}
  \end{theorem}

  The error bound in Theorem \ref{errbound2} consists of three parts:
  statistical error $\|\widetilde{\varepsilon}_{\!S^{*}}\|$ induced by noise,
  the identification error
  $\sum_{j=1}^{k-1}\!\sqrt{{\textstyle\sum_{i\in S^{*}}}\mathbb{I}_{\Lambda^j}(i)}(\frac{1}{\sqrt{2}})^{k-1-j}$
  related to the choice of $\rho_{j}$, and the computation error
  $(\frac{1}{\sqrt{2}})^{k-1}\|\beta^{1}\!-\beta^{*}\|$.
  By the definition of $\Lambda^{j}$, if $\rho_{j}$ is chosen
  such that $\rho_{j}>\frac{4a}{(a+1)\min_{i\in S^*}\!|\beta_i^*|}$,
  then the identification error becomes zero, and consequently
  the error bound sequence will decrease to the statistical error
  $\|\widetilde{\varepsilon}_{\!S^{*}}\|$ as $k$ increases. Clearly,
  if $\min_{i\in S^*}\!|\beta_i^*|$ is not too small,
  it is easy to choose such $\rho_{j}$.
  In the next part, we shall provide an explicit choice range
  of $\rho_j$ such that the identification error is zero.
  From Theorem \ref{errbound2}, we also observe that a smaller error
  bound of $\beta^1$ brings a smaller error bound for $\beta^k$
  with $k\ge 2$. The importance of $\beta^1$ also comes from the fact
  that one may use it to estimate the choice range of $\rho_j\ (j\ge 1)$
  since $\|\widetilde{\varepsilon}\|_{\infty}$ is unknown in practice.
  During the implementation of GEP-MSCRA,
  we choose $\rho_1$ according to this strategy.

%--------------------------------------------------------------------------
 \noindent {\bf 4.2. Sign consistency}\label{sec4.2}

 We show that if the smallest nonzero component of $\beta^{*}$ is not so small,
 GEP-MSCRA can deliver $\beta^l$ satisfying ${\rm supp}(\beta^l)={\rm supp}(\beta^{*})$
 within finite steps. To achieve this goal, we need the oracle
 least squares solution:
 \begin{equation}\label{defxls}
  \beta^{{\rm LS}}\in\mathop{\arg\min}_{\beta\in\mathbb{R}^p}\left\{\frac{1}{2n}\|\widetilde{Z}\beta-\widetilde{y}\|^{2}:\
   {\rm supp}(\beta)\subseteq S^{*}\right\}.
 \end{equation}
 Write $\varepsilon^{{\rm LS}}:=\frac{1}{n}\widetilde{Z}^{\mathbb{T}}(\widetilde{y}-\widetilde{Z}\beta^{\rm LS})$.
 Then
 \(
   \varepsilon_{S^*}^{{\rm LS}}=\widetilde{Z}_{\!S^{*}}^\mathbb{T}(\widetilde{Z}\beta^{{\rm LS}}-\widetilde{y})=0.
 \)
 This implies that
 \(
   \beta^{\rm LS}_{S^{*}}-\beta^{*}_{S^{*}}
  =\widetilde{\Sigma}_{\!S^*S^*}^{-1}\big[\frac{1}{n}\widetilde{Z}^{\mathbb{T}}_{\!S^{*}}
   \big(\widetilde{Z}\beta^{\rm LS}-\widetilde{Z}\beta^{*}\big)\big]
  =\widetilde{\Sigma}_{\!S^*S^*}^{-1}\big[\frac{1}{n}\widetilde{Z}^{\mathbb{T}}_{S^{*}}(\widetilde{y}-\widetilde{Z}\beta^{*})\big],
 \)
 and
 \begin{equation}\label{sign-ls}
 \beta^{\rm LS}_{S^{*}}-\beta^{*}_{S^{*}}
  =\widetilde{\Sigma}_{\!S^*S^*}^{-1}\big(\widehat{\xi}_{S^{*}}-\widetilde{\Sigma}_{\!S^{*}S^{*}}\beta^{*}_{S^{*}}\big)
  =\widetilde{\Sigma}_{\!S^*S^*}^{-1}\widetilde{\varepsilon}_{\!S^{*}}:=\widetilde{\varepsilon}^\dag.
 \end{equation}
 Based on this observation for $\beta^{\rm LS}$,
 we can establish the following result.
%------------------------------------------------------------------------Theorem3.3
 \begin{theorem}\label{sign-consistency}
  Suppose $\Sigma$ satisfies the $\kappa$-REC on $\mathcal{C}(S^{*})$
  with $\kappa>24s\|D\|_{\rm max}$.
  Set $\gamma:=\kappa-24s\|D\|_{\rm max}$.
  If $\lambda,\rho_1$ and $\rho_3$ are chosen
  with $\lambda\ge 6\|\varepsilon^{\rm LS}\|_\infty$,
  $\rho_1\!>\!\max\!\big(\frac{4a}{(a+1)\min_{i\in S^*}\!|\beta^{*}_{i}\!|},
  \gamma\lambda^{-1}\|\widetilde{\varepsilon}^\dag\|_{\infty}\big)$
  and $\rho_3\le\!\sqrt{\frac{4\gamma}{9\sqrt{3}\lambda}}$, then for all $k\in\mathbb{N}$
  \begin{equation*}
  \big\|\beta^k-\beta^{\rm{LS}}\big\|
  \le \frac{2.03\rho_{k-1}\lambda}{\gamma}\sqrt{|F^{k-1}|},\,
  \sqrt{|F^{k}|}\le\frac{18.27\sqrt{3}\rho_k\rho_{k-1}\lambda}{(9\sqrt{3}\!-\!4)\gamma}\sqrt{|F^{k-1}|}.
  \end{equation*}
  In particular, when $k\ge \overline{k}$ with
  $\overline{k}=\big\lceil\frac{0.5\ln(s)}{\ln[(9\sqrt{3}\!-\!4)\gamma\lambda^{-1}]
  -\ln[18.27\sqrt{3}(\rho_3)^2]}\big\rceil$, we have
  \[
    \beta^{k}=\beta^{\rm LS}\ \ {\rm and}\ \ {\rm sign}(\beta^{k})={\rm sign}(\beta^{*}).
  \]
 \end{theorem}
 \begin{remark}\label{sign-remark}
 {\bf(a)} Notice that \cite{Datta16} achieved the sign consistency
 of $\overline{\beta}$ under an irrepresentable condition on $\Sigma$
 and the condition $\min_{i\in S^*}\!|\beta^{*}_{i}\!|>\![4\|\Sigma_{S^*S^*}^{-1}\|_{\infty}
  +(\lambda_{\min}(\Sigma_{S^*S^*}))^{-1/2}]\lambda$, where
  $\|A\|_\infty=\max_{i}\sum_{j}|A_{ij}|$ means the matrix $\ell_\infty$-norm.
  Their irrepresentable condition on $\Sigma$ requires that
  $\|\Sigma_{(S^*)^cS^*}\Sigma_{S^*S^*}^{-1}\|_\infty\!\le\overline{\gamma}<1$
  and $\lambda_{\rm min}(\Sigma_{S^*S^*})\ge\!C_{\rm min}$ for some constants
  $\overline{\gamma}>0$ and $C_{\rm min}\!>0$, in which the former
  makes a restriction on the scale of the entries of $\Sigma$
  and the latter is precisely the REC of $\Sigma$ on the set
  $\big\{\beta\in \mathbb{R}^p\!: \beta_{(S^*)^{c}}=0\big\}$.
  We obtain the sign consistency of $\beta^k$ for $k\ge\overline{k}$
  under the $\kappa$-REC of $\Sigma$ on $\mathcal{C}(S^*)$
  with $\kappa>24s\|D\|_{\rm max}$ and
  $\rho_1\!>\!\max\!\big(\frac{4a}{(a+1)\min_{i\in S^*}\!|\beta^{*}_{i}\!|},  \gamma\lambda^{-1}\|\widetilde{\varepsilon}^\dag\|_{\infty}\big)$.
  When $\lambda_{\rm min}(\Sigma_{S^*S^*})$ is large,
  there is a great possibility for our $\kappa$-REC to hold.
  Also, when $\|\Sigma_{(S^*)^cS^*}\Sigma_{S^*S^*}^{-1}\|_{\infty}
  \le\overline{\gamma}$ does not hold, our $\kappa$-REC may hold;
  for example, consider $\Sigma=[1\ 0\ 2; 0\ 1\ 2; 2\ 2\ 9]$ and $S^*=\{1,2\}$.
  In fact, to some extent, our $\kappa$-REC also
  depends on the unbiased surrogate $\widehat{\Sigma}$ of $\Sigma$.
  If $\|\widehat{\Sigma}-\Sigma\|_{\rm max}$ is small, there is a great
  possibility for our $\kappa$-REC to hold. Finally, the condition on
  $\min_{i\in S^*}\!|\beta^{*}_{i}\!|$ used by \cite{Datta16} implies a large
  choice range for our parameter $\rho_1$ whether $\|\Sigma_{S^*S^*}^{-1}\|_{\infty}$
  or $(\lambda_{\min}(\Sigma_{S^*S^*}))^{-1/2}$ is larger
  or $\lambda$ is larger.

  \noindent
  {\bf(b)} Notice that $\rho_3\le\!\sqrt{\frac{4\gamma}{9\sqrt{3}\lambda}}$.
  Together with the definition of $\overline{k}$, we have
  $\ln[(9\sqrt{3}\!-\!4)\gamma\lambda^{-1}]-\ln[18.27\sqrt{3}(\rho_3)^2]\ge\ln(1.4)$,
  which with $s\ge 9$ implies that
  \(
   \overline{k}\le\widehat{k}:=\big\lceil\frac{0.5\ln(s)}{\ln(1.4)}\big\rceil.
  \)
  As one referee pointed out that $\overline{k}$ or $\widehat{k}$ is actually
  unknown since it depends on the sparsity $s$ of $\beta^*$.
  In practice, some prior upper estimation on $s$ is available;
  for example, a rough upper estimation on $s$ is the dimension $p$.
  Thus, one still can obtain a rough upper estimation on $\widehat{k}$.
  In the practical numerical computation,
  one can identify such $\overline{k}$ by monitoring the index change
  of nonzero entries in each iterate.

  \noindent
  {\bf(c)} By Theorem \ref{sign-consistency},
  the choice of $\rho_1$ is crucial for GEP-MSCRA to yield an oracle solution
  whose sign is consistent with that of $\beta^*$ after finite steps.
  As remarked after Theorem \ref{errbound2}, whether such $\rho_1$ is easily
  chosen or not depends on the error bound of $\beta^{1}$.
  From Theorem \ref{sign-consistency} and Theorem \ref{errbound2},
  we conclude that a smaller $\rho_3$ entails a good output of GEP-MSCRA
  in terms of the error bound and sign consistency, and for those problems
  with high noise, a large $\lambda$ is needed, and of course the error bound
  of $\beta^k$ becomes large.
 \end{remark}
  We have established the deterministic theoretical guarantees of GEP-MSCRA
  for computing the calibrated zero-norm regularized LS estimator under suitable
  conditions. From \citep{Raskutti10,Raskutti11}, if $X$ is from
  the $\Sigma_x$-Gaussian ensemble (i.e., $X$ is formed by independently sampling
  each row $X^i\sim N(0,\Sigma_x)$, there exists a constant $\kappa>0$
  (depending on $\Sigma_x$) such that $\Sigma$ satisfies the REC on
  $\mathcal{C}(S^*)$ with probability greater than $1\!-\!c_1\exp(-c_2n)$
  as long as $n>cs\ln p$, where $c,c_1$ and $c_2$ are absolutely positive constants.
  It is natural to ask whether such $\kappa$ satisfies the requirement of
  the above theorems or not. Is there a big possibility to choose
  $\lambda,\rho_1$ and $\rho_3$ as required in the above theorems?
  In Appendix B, we focus on these questions for two specific types
  of errors-in-variables models.
%------------------------------------------------------------------------
 \section{Numerical experiments}\label{sec5}

  We use simulated datasets to evaluate the performance of the CaZnRLS estimator,
  computed with GEP-MSCRA (see Appendix C for its implementation details),
  and compare its performance with that of CoCoLasso and NCL in terms of
  the number of signs identified correctly (NC) and identified incorrectly (NIC)
  for predictors, and the relative root-mean-square-error (RMSE).
  Let $\beta^{f}$ be the final output of one of three solvers. Define
  \begin{align*}
  {\rm NC}(\beta^{f}):=\!\sum_{i\in S^*}\mathbb{I}\big\{|{\rm sign}(\beta_i^{f})-{\rm sign}(\beta_i^*)|=0\big\},\qquad\qquad\\
   {\rm NIC}(\beta^{f}):=N_{\rm nz}(\beta^f)-{\rm NC}(\beta^{f})\ \ {\rm and}\ \
   {\rm relative\ RMSE}:=\frac{\|\beta^{f}-\beta^{*}\|}{\|\beta^*\|}
  \end{align*}
  where $N_{\rm nz}(\beta^f)\!:=\!\sum_{i=1}^p\mathbb{I}\big\{|\beta_i|>\!10^{-8}\big\}$
  is the number of nonzero entries of $\beta^f$.
  All results are obtained in a desktop computer running on 64-bit Windows
  with an Intel(R) Core(TM) i7-7700 CPU 3.6GHz and 16 GB memory.

  For GEP-MSCRA, we choose $a=6.0$ for $\phi$, $w^0=0$ and $\rho_k$ for $k\le 3$ by
  \[
    \rho_1=\max\Big(1,\frac{5}{3\|\beta^1\|_{\infty}}\Big),\
    \rho_k=\min\Big(2\rho_{k-1},\frac{10^8}{\|\beta^k\|_\infty}\Big)\ {\rm for}\ k=2,3.
  \]
  We terminate GEP-MSCRA at $\beta^k$ once the following condition is satisfied
  \[
    \left\{\begin{array}{c}
    |N_{\rm nz}(\beta^{k-j})-N_{\rm nz}(\beta^{k-j-1})|\le 5,\ j=0,1,2;\\
    \big|\frac{1}{2n}\|\widetilde{Z}\beta^k\!-\!\widetilde{y}\|^2-
  \frac{1}{2n}\|\widetilde{Z}\beta^{k-1}\!-\!\widetilde{y}\|^2\big|\le 0.1,
    \end{array}\right.
  \]
  or the number of iterates is over the maximum number $k_{\rm max}=4$
  (Our code can be achieved from \url{https://github.com/SCUT-OptGroup/ErrorInvar}).
  Such a stopping criterion aims to capture a solution
  $\beta^k$ whose sparsity tends to be stable on one hand, and on the other hand
  its predictor error has a small variation. In addition,
  by Remark \ref{sign-remark}(b), we have a rough upper estimation
  for $\overline{k}$ is $\lceil\frac{0.5\ln(p)}{\ln(1.4)}\rceil$,
  which equals $11$ for $p=1000$. In view of this, we set the maximum
  number of iterate to be $4$.  We solve the dual of \eqref{expm-subx}
  with Algorithm \ref{SNAL} for $\epsilon^j=10^{-8}$. For NCL,
  we directly run the code ``doProjGrad'', which is solving
  the model \eqref{CLASSO} with $\lambda_n=0$ and $R_0=\|\beta^*\|_1$,
  for the test examples.
  Since the Matlab code of CoCoLasso is not available,
  we include our implementation in Appendix D. Since
  it is time-consuming for Algorithm \ref{ADMM-alg} to use the stopping rule
  $\max\{\epsilon_{{\rm pinf}}^k,\epsilon_{{\rm dinf}}^k,\epsilon_{{\rm gap}}^k\}\!\le 10^{-5}$,
  we use the looser
  \(
  \max\{\epsilon_{{\rm pinf}}^k,\epsilon_{{\rm dinf}}^k,10^{-3}\epsilon_{{\rm gap}}^k\}\!\le 10^{-4}
  \)
  to get an approximate solution of \eqref{Convex1}, and then
  use Algorithm \ref{SNAL} to solve the associated problem \eqref{CoCoLasso}.

  From the theoretic results in Section \ref{sec4}, the appropriate $\lambda$
  lies in an interval associated to $\|\widetilde{\varepsilon}\|_{\infty}$.
  Such $\lambda$ is also suitable for CoCoLasso by the proof of Theorem 1
  and 2 in \cite{Datta16}. In view of this, we set $\lambda=\max(0.01,\frac{\alpha^*}{n}\|\widetilde{Z}^{\mathbb{T}}\widetilde{y}\|_{\infty})$
  and $\max(0.01,\frac{\alpha^*}{n}\|\overline{Z}^{\mathbb{T}}\overline{y}\|_{\infty})$
  for CaZnRLS and CoCoLasso, respectively, where the appropriate $\alpha^*\in[0.06,0.32]$
  is chosen by using $5$-fold corrected cross-validation proposed in \cite{Datta16}.

  Throughout this section, all test examples are generated randomly with
  the triple $(p,s,n)$ consisting of the dimension $p$ of predicted variable,
  the number of nonzero entries of $\beta^*$, and the sample size $n$.
  Among others, $n=\lfloor\alpha s\ln(p)\rfloor$ with
  $\alpha=4+0.2(j\!-\!1)$ for $j=1,\ldots,11$.
  We obtain the observation $y$ from the model \eqref{observation}
  where the entries of $\varepsilon$ are i.i.d. $\mathcal{N}(0,\sigma^2)$,
  and the generating way of the true $\beta_{S^*}^*$ is specified in the sequel.
  The average relative RMSE (respectively, NC and NIC) is the average of
  the total RMSE (respectively, NC and NIC) for $100$ problems generated randomly.

 %-------------------------------------------------------------------------------------------
  \noindent
  {\bf 5.1. Random locations of the nonzero entries of $\beta^*$}\label{sec5.1}

  In this part, we evaluate the performance of CaZnRLS by the examples
  generated randomly, where $\beta_{S^*}^*$ is an i.i.d. standard normal
  random vector with the $s=\lfloor0.5\sqrt{p}\rfloor$ entries of $S^*$
  chosen randomly from $\{1,\ldots,p\}$. First, we test whether CaZnRLS
  is stable with respect to the variance $\sigma$ of $\varepsilon$.
%------------------------------------------------------------------------------
 \noindent
 {\bf 5.1.1. Performance of CaZnRLS under low and high noise}\label{sec5.1}

%--------------------------------------------------------------------------
 \begin{example}\label{example1}
  We generate $Z=X\!+A$ with $p=500$, where the rows of $X$ are i.i.d.
  standard normal random vectors with mean zero and covariance matrix $\Sigma_X=I$,
  and the rows of $A$ are i.i.d. $\mathcal{N}(0,I)$.
 \end{example}

  Figure \ref{fig-example1} plots the average relative RMSE, NC and NIC curves
  of CaZnRLS, CoCoLasso and NCL for Example \ref{example1} under different sample
  sizes with $\sigma=0.5$ and $1.0$. The subfigures on the first column show that
  CaZnRLS is comparable even a little better than CoCoLasso in terms of the relative RMSE,
  the second column shows that the NC of CaZnRLS is at most $2$ fewer than
  that of CoCoLasso, while the third column indicates that the NIC of CaZnRLS
  is much fewer than that of CoCoLasso. From this, we conclude that CaZnRLS keeps
  the merit of the zero-norm regularized LS estimator in the clean data setting.
  We also see that CaZnRLS has a similar performance for $\sigma=0.5$ and $\sigma=1$,
  indicating that it is insensitive to the variance $\sigma$ of regression error.
  So, in the sequel we always take $\sigma=0.5$.
  \begin{figure}[H]
  \setlength{\abovecaptionskip}{-0.2cm}
  \centerline{\epsfig{file=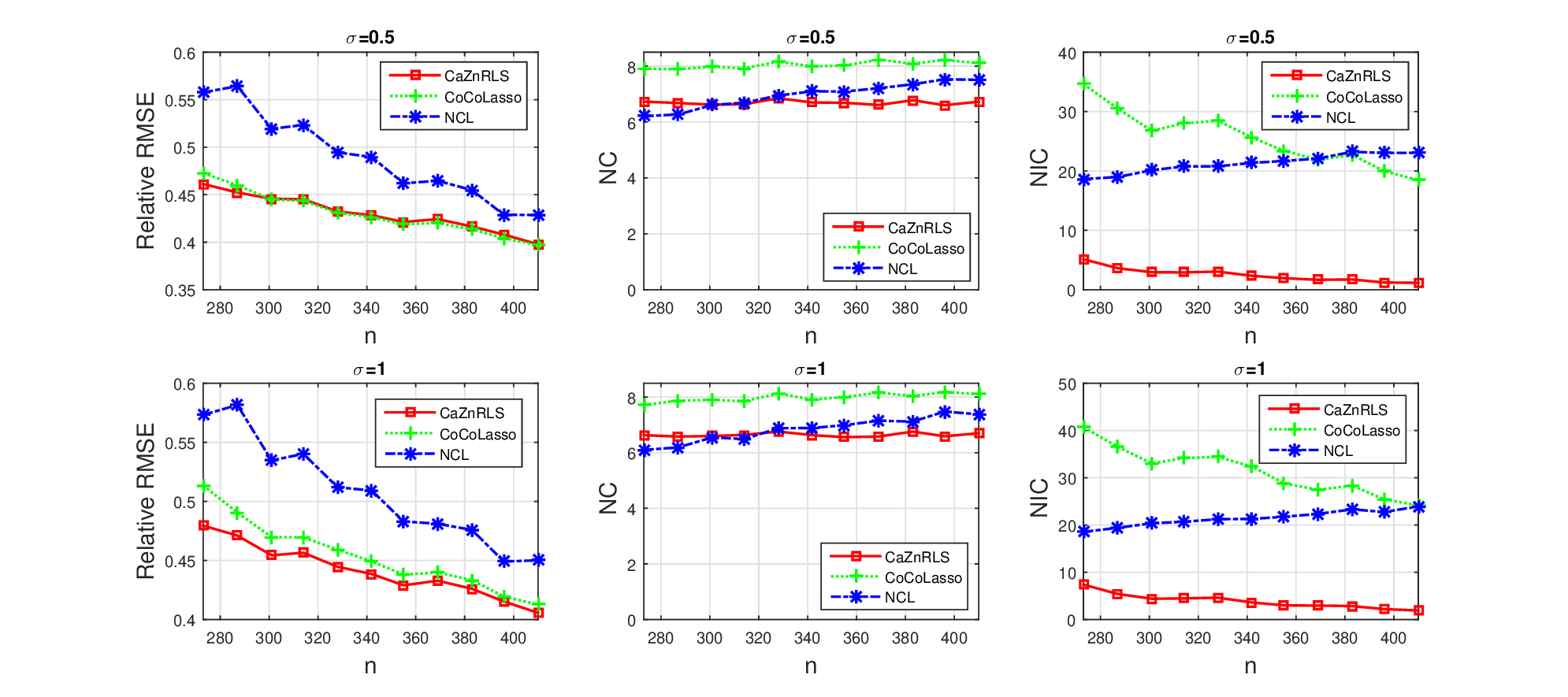,width=4.5in}}\par
 \caption{\small The relative RMSE, NC and NIC of three solvers under different $\sigma$}
 \label{fig-example1}
 \end{figure}

 %------------------------------------------------------------------------------
 \noindent {\bf 5.1.2. Performance of CaZnRLS for different measurement errors}\label{sec5.2}

  In this part we evaluate the performance of CaZnRLS for three classes of measurement errors
  by using test problems generated with $p=1000$.

%--------------------------------------------------------------------------
 \noindent
 {\bf Case 1. Additive errors}\label{sec5.2.1}
%--------------------------------------------------------------------------------------
 \begin{example}\label{example2}
   We generate $Z=X+A$ where $X$ is same as in Example \ref{example1},
  and the rows of $A$ are i.i.d. $\mathcal{N}(0,\tau^2I)$ with $\tau=0.5$ or $1.0$.
 \end{example}
 \begin{example}\label{example3}
  We generate $Z=X+A$ where the entries of $X$ are i.i.d. and
  follow the uniform distribution on $(0,1)$, and $A$ is same as
  in Example \ref{example2}.
 \end{example}

 \begin{figure}[H]
  \setlength{\abovecaptionskip}{-0.2cm}
  \centerline{\epsfig{file=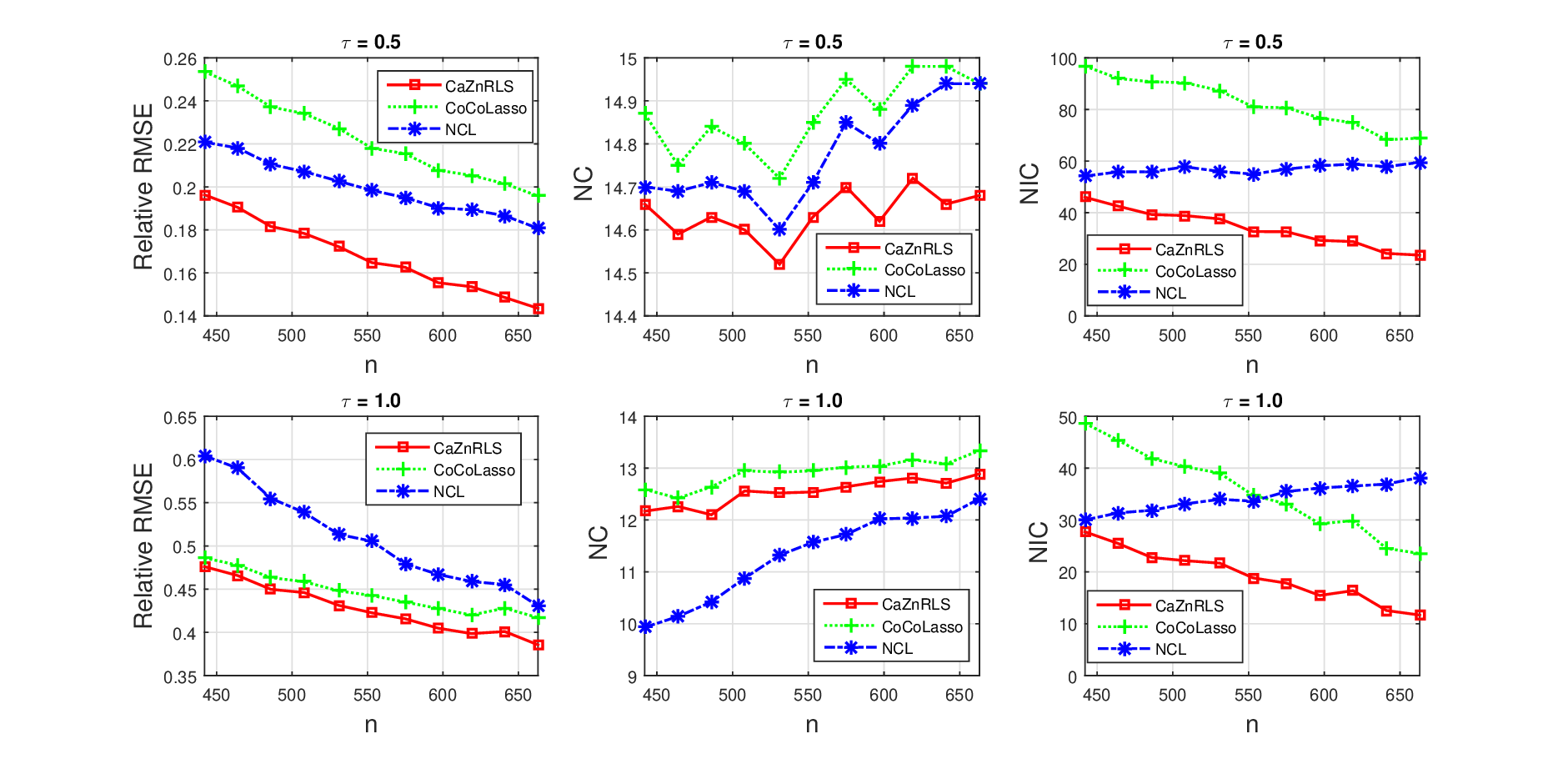,width=4.5in}}\par
  \caption{\small The relative RMSE, NC and NIC of three solvers under different $n$}
  \label{fig-example2}
 \end{figure}

  Figure \ref{fig-example2} plots the average relative RMSE, NC and NIC curves
  of three solvers under different sample sizes for Example \ref{example2}.
  From this figure, whether $X$ is corrupted by high noise or low noise,
  CaZnRLS is the best among three solvers in terms of the relative RMSE and NIC,
  though its NC is (at most $1$) fewer than the NC of CoCoLasso. The relative RMSE
  of CaZnRLS improves that of CoCoLasso at least $20\%$ for the low noise,
  and $4\%$ for the high noise when $n\ge \lfloor5s\ln(p)\rfloor$.
  We also see that NCL has the worst performance in terms of the relative RMSE, NC
  and NIC for the high noise.
 \begin{figure}[H]
  \setlength{\abovecaptionskip}{-0.2cm}
 \centerline{\epsfig{file=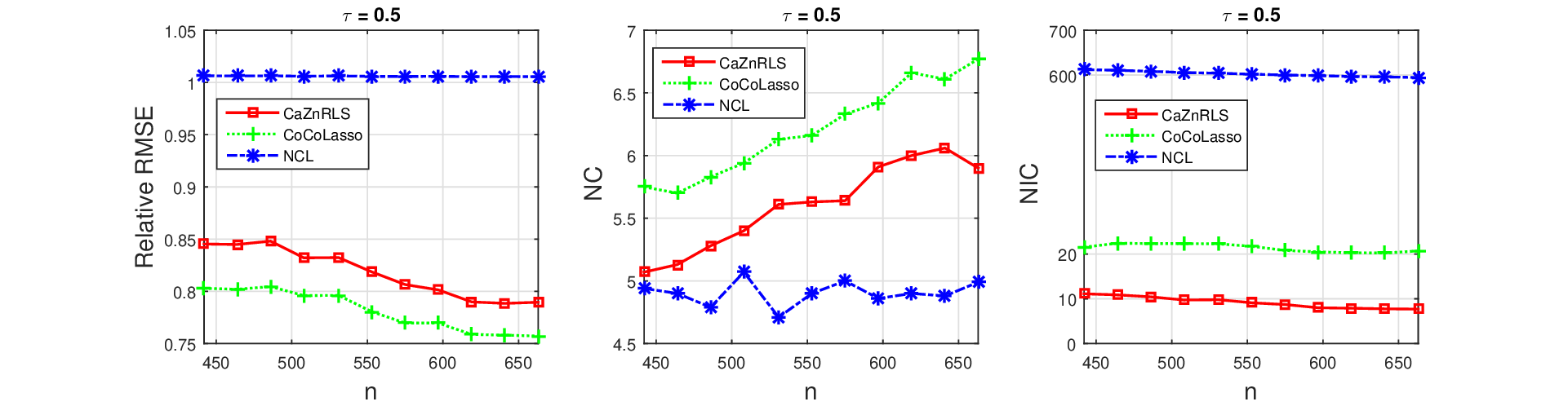,width=4.5in}}\par
 \caption{\small The relative RMSE, NC and NIC of three solvers under different $n$}
 \label{fig-example3}
 \end{figure}

  Figure \ref{fig-example3} plots the average relative RMSE, NC and NIC curves
  of three solvers under different sample sizes for Example \ref{example3}.
  We see that three solvers have much higher relative RMSE than they do
  for Example \ref{example2}, and NCL even fails in giving the desired estimator.
  The relative RMSE of CaZnRLS is a little (about $4\%$) higher than that of CoCoLasso.
  After checking the unbiased estimation $\Sigma$ of the covariance matrix
  of the true covariates, we find that the irrepresentable and minimum eigenvalue conditions
  in \cite{Datta16} are not satisfied. Now it is not clear whether our REC on
  $\mathcal{C}(\beta^*)$ holds or not. This does not contradict to the theoretical analysis
  in Section \ref{sec4} since now it is only known that our REC on $\mathcal{C}(\beta^*)$
  holds w.h.p. when $X$ is from Gaussian ensemble. The first subfigure
  indicates that it is very likely for our REC not to hold when $X$ is from the uniform
  distribution.

%------------------------------------------------------------------------------
 \noindent
 {\bf Case 2. Multiplicative errors}\label{sec5.2.2}
%-------------------------------------------------------------------------------
 \begin{example}\label{example4}
  We generate $Z=X\circ M$ where the rows of $X$ are i.i.d. $\mathcal{N}(0,I)$,
  and the entries of $M$ are i.i.d. and follow the log-normal distribution,
  i.e., $\ln(M_{ij})$ are i.i.d and follow $N(0,\tau^{2}I)$
  with $\tau=0.5$ or $0.8$.
 \end{example}
 \begin{example}\label{example5}
  We generate $Z$ in the same way as in Example \ref{example4}
  except that the entries of $X$ are i.i.d. and follow the Laplace distribution
  of mean 0 and variance 1.
 \end{example}

  \begin{figure}[H]
  \setlength{\abovecaptionskip}{-0.2cm}
   \centerline{\epsfig{file=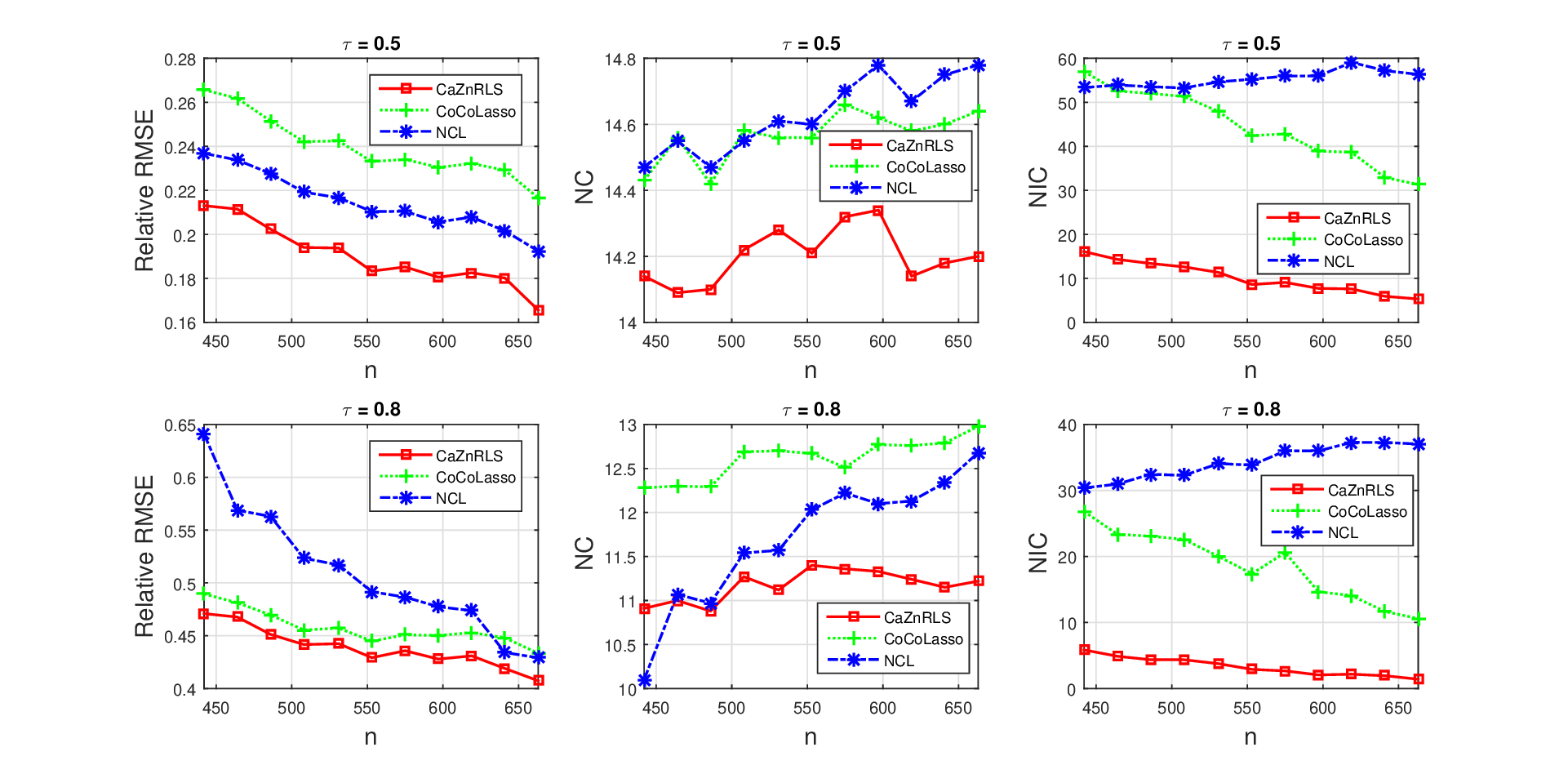,width=4.5in}}\par
  \caption{\small The relative RMSE, NC and NIC of three solvers under different $n$}
 \label{fig-example4}
 \end{figure}
 \begin{figure}[H]
  \setlength{\abovecaptionskip}{-0.2cm}
   \centerline{\epsfig{file=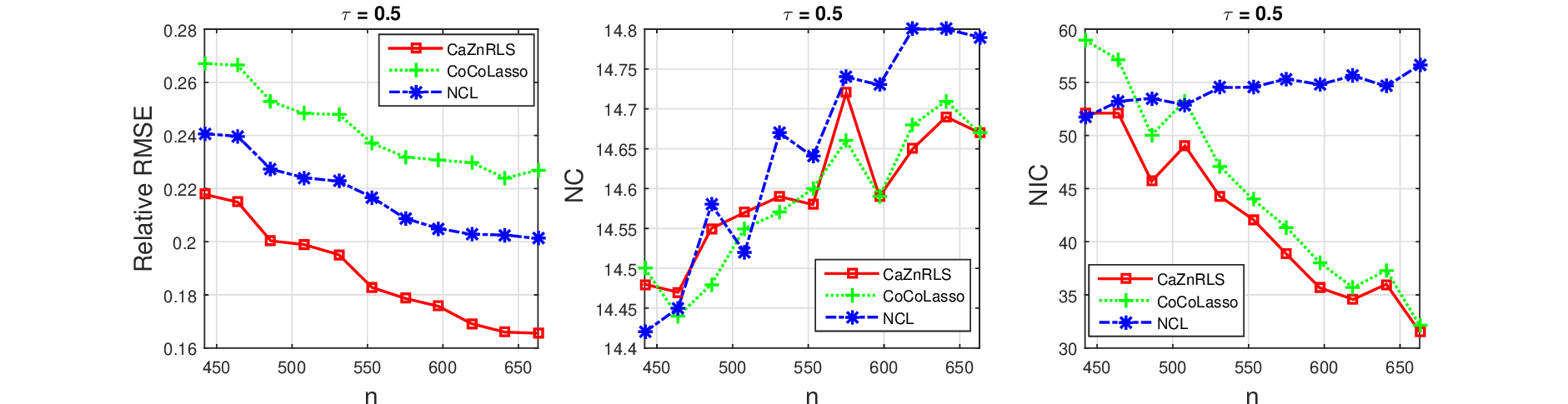,width=4.5in}}\par
  \caption{\small The relative RMSE, NC and NIC of three solvers under different $n$}
 \label{fig-example5}
 \end{figure}

 Figure \ref{fig-example4} and \ref{fig-example5} plot the average relative RMSE,
 NC and NIC curves of three solvers under different $n$ for
 Example \ref{example4} and \ref{example5}, respectively.
 By comparing Figure \ref{fig-example4} with Figure \ref{fig-example2},
 we see that CaZnRLS and CoCoLasso have the similar performance
 as they do for the additive errors. That is, CaZnRLS is better than
 CoCoLasso in terms of the relative RMSE and NIC whether for $X$
 corrupted by high noise or low noise, though its NC is (at most $2$) fewer
 than the NC of CoCoLasso. Along with Figure \ref{fig-example5},
 we conclude that CaZnRLS has the similar performance when the rows of $X$ follow
 the Gaussian and Laplace distribution.

%------------------------------------------------------------------------------
 \noindent
 {\bf Case 3. Missing data case}\label{sec5.2.3}
%---------------------------------------------------------------------------------

 \begin{example}\label{example6}
  We generate $(Z_{ij})_{n\times p}$ for $Z_{ij}=X_{ij}$
  with probability $1-\tau$ and $Z_{ij}=0$ with probability $\tau$
  for $\tau=0.3$ or $0.5$, where the rows of $X$ are i.i.d. and obey
  the standard normal distribution $\mathcal{N}(0,I)$.
  \end{example}
 \begin{example}\label{example7}
  We generate $Z$ in the same way as in Example \ref{example6} except
  that $X_{ij}$ are i.i.d. and obey the exponential
  with mean 1 and variance 1.
 \end{example}
 \begin{figure}[H]
 \setlength{\abovecaptionskip}{-0.2cm}
  \centerline{\epsfig{file=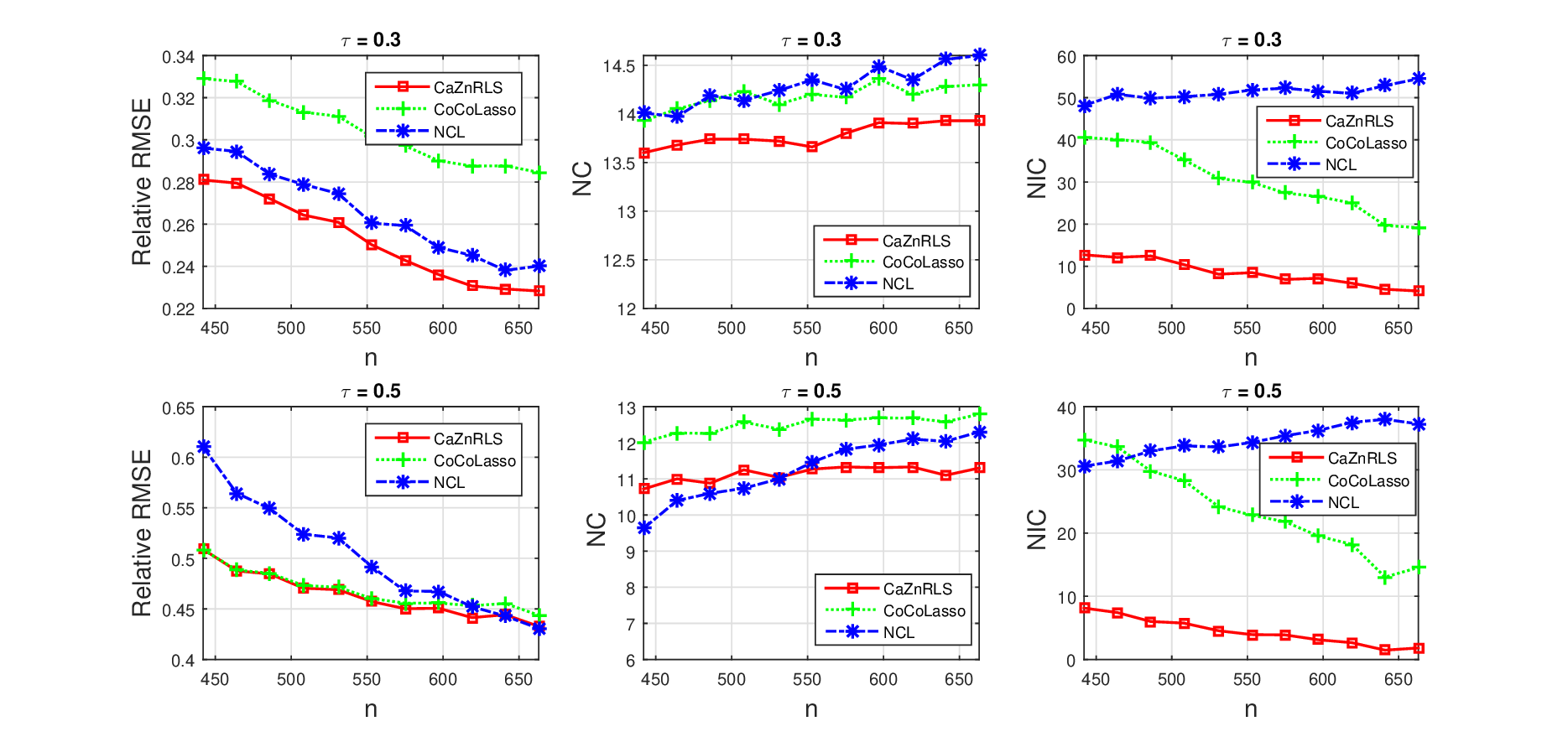,width=4.5in}}\par
 \caption{\small The relative RMSE, NC and NIC of three solvers under different $n$}
 \label{fig-example6}
 \end{figure}
  \begin{figure}[H]
  \setlength{\abovecaptionskip}{-0.2cm}
  \centerline{\epsfig{file=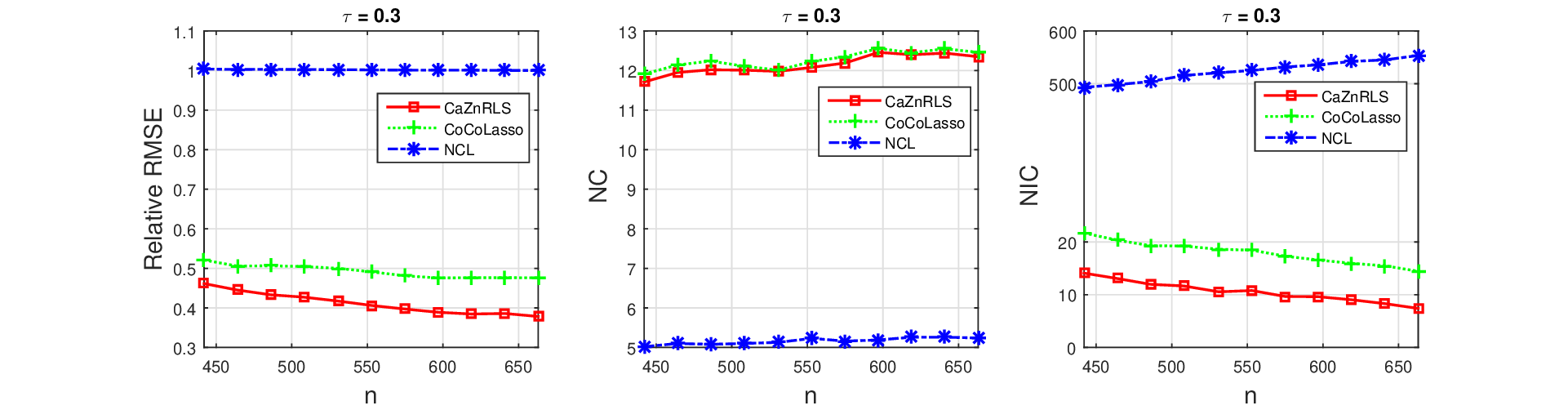,width=4.5in}}\par
 \caption{\small The relative RMSE, NC and NIC of three solvers under different $n$}
 \label{fig-example7}
 \end{figure}

 Figure \ref{fig-example6} and \ref{fig-example7} plot the average relative RMSE,
 NC and NIC curves of three solvers under different $n$ for Example \ref{example6}
 and \ref{example7}, respectively. Comparing Figure \ref{fig-example6} with
 Figure \ref{fig-example2} or \ref{fig-example4}, we see that the three solvers
 have similar performance as they do for the additive and multiplicative errors.
 In fact, we check that similar to Example \ref{example2} and
 \ref{example4}-\ref{example5}, Example \ref{example6} satisfies the irrepresentable
 and minimum eigenvalue conditions in \cite{Datta16} when $n\ge \lfloor4.4s\ln(p)\rfloor$.
 Of course, our REC on $\mathcal{C}(\beta^*)$ holds with a high probability for
 Example \ref{example2} and \ref{example4}, and Figure \ref{fig-example5}-\ref{fig-example7}
 also indicate that our REC holds with a high probability when the rows of $X$ follow
 the Laplace and exponential distributions. Figure \ref{fig-example7} shows that,
 when the entries of $X$ follow the exponential distribution, CaZnRLS is superior to
 the other two solvers in terms of the relative RMSE and NIC, and its RMSE improves
 that of CoCoLasso at least $11\%$. Now NCL fails in yielding the desired estimator.
 After checking, we find that Example \ref{example7} actually does not satisfy
 the irrepresentable and minimum eigenvalue conditions in \cite{Datta16}.
 Now it is not clear whether our REC holds or not for this example.

  Motivated by one referee's comments, we next provide an example
  that does not satisfy the irrepresentable condition but our REC holds w.h.p..
  \begin{example}\label{example8}
   We generate $Z=X+A$ with $p=250$ where the entries of $X_{S^*}$ are
   i.i.d. $\mathcal{N}(0,1)$, the entries of $X_{(S^*)^{c}}$ are i.i.d.
   $\mathcal{N}(0,5^2)$, and the rows of $A$ are generated in the same way
   as in Example \ref{example2} with $\tau=0.75$.
  \end{example}

  Figure \ref{fig-example8} plots the average relative RMSE,
  NC and NIC curves of CaZnRLS and CoCoLasso under different $n$
  for Example \ref{example8}. Since NCL fails in this example,
  we do not include its results in Figure \ref{fig-example8}.
  We see that the relative RMSE of CaZnRLS is lower than that of CoCoLasso,
  and when $n\ge \lfloor5s\ln(p)\rfloor$, the relative RMSE of CaZnRLS
  improves at least $10\%$ that of CoCoLasso.
  The NC and NIC of CoCoLasso are still higher than those of CaZnRLS,
  but NC of the latter is at most $1$ lower than that of the former.
  This example further confirms the theoretical results in Section \ref{sec4}.
  \begin{figure}[H]
  \setlength{\abovecaptionskip}{-0.2cm}
  \centerline{\epsfig{file=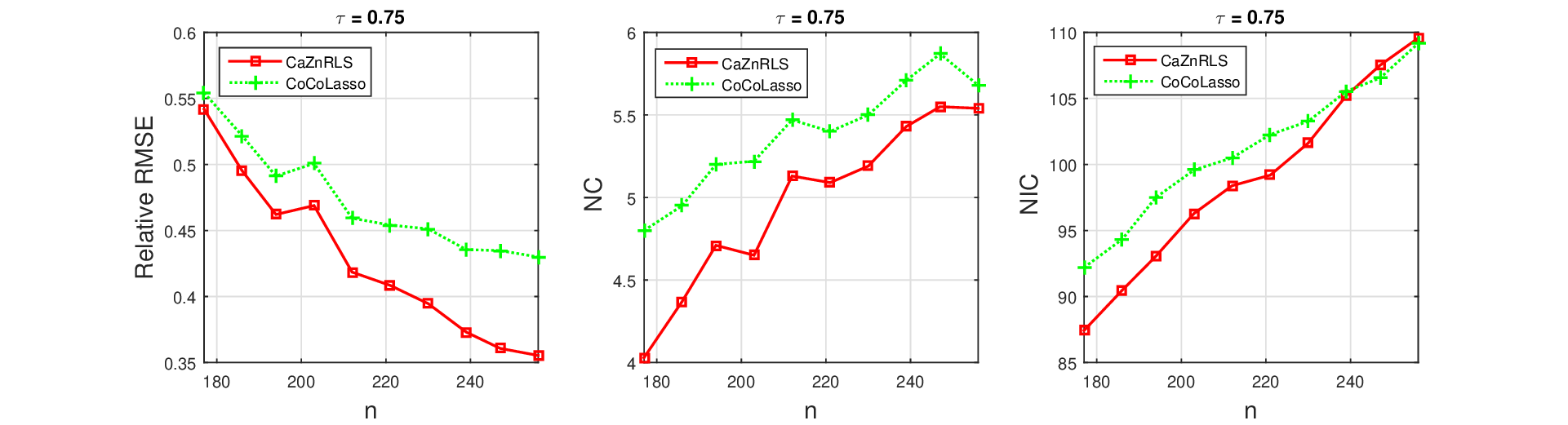,width=4.5in}}\par
 \caption{\small The relative RMSE, NC and NIC of three solvers under different $n$}
 \label{fig-example8}
 \end{figure}

%--------------------------------------------------------------------------------------------
  \noindent
  {\bf 5.2. Fixed locations of the nonzero entries of $\beta^*$}\label{sec5.2}
 %-------------------------------------------------------------------------------------------

  As one referee pointed out, it would be interesting to show
 the effects of the correlation between the predictors on the performance of
 three solvers. In this part, we test whether the correlation between
 the predictors has an effect on the performance of three solvers or not,
 by using the examples generated by \cite{Datta16} in which the locations of
  the nonzero entries of $\beta^*$ are fixed. Specifically,
  $\beta^*=(3,1.5,0,0,2,0,\ldots,0)$ with the number of nonzero entries $s=3$.
  The data $X$ is generated with $p=250$ and $n=100$ such that the rows of $X$
  obey i.i.d. $\mathcal{N}(0,\Sigma_{X})$ for $(\Sigma_{X})_{ij}=0.5^{|i-j|}$.
  Table \ref{table1} summaries the simulation results of three solvers
  for additive errors, multiplicative errors and missing data, where
  the error matrices $A$ and $M$ for the additive and multiplicative errors
  are generated in the same way as in Example \ref{example2} and \ref{example4}
  respectively, while the contaminated matrix $Z$ in missing data is generated in
  the same way as in Example \ref{example6}.

 \begin{table}[H]\tiny
 \setlength{\abovecaptionskip}{2pt}
 \setlength{\belowcaptionskip}{0pt}
 \centering
 \caption{The average relative RMSE, NC and NIC of three solvers}\label{table1}
 \tiny
 \begin{tabular}{|c|ccc|ccc|ccc|}
 \hline
 \raisebox{2.0ex}[15pt]{}
 \multirow{4}*{}
  &\multicolumn{3}{c|}{\bf Additive errors } & \multicolumn{3}{c|}{\bf Multiplicative errors }
  &\multicolumn{3}{c|}{\bf Missing data } \\
  &\multicolumn{3}{c|}{\bf $\tau=1$} & \multicolumn{3}{c|}{\bf  $\tau=0.8$} &\multicolumn{3}{c|}{\bf $\tau=0.5$} \\
  \cline{2-10}
  \!&\!CaZnRLS\!&\!CoCoLasso\!&\!NCL\!&\!CaZnRLS\!&\!CoCoLasso\!&\!NCL\!&\!CaZnRLS\!&\!CoCoLasso&\!NCL\\
 \hline
 RMSE& 0.410 &0.492 &0.535 & 0.370&0.524 &0.600 &0.447 & 0.521&0.528\\
 \hline
 NC & 2.81 &2.87 &2.41 &2.76 &2.87 &2.18 &2.69 &2.75 &2.27\\
 \hline
 NIC& 1.48& 2.46& 6.48 &1.30 &2.48 &5.31 &2.41 &2.60 &6.90\\
 \hline
 \end{tabular}
 \end{table}

 From Table \ref{table1}, CaZnRLS yields the lowest relative
 RMSE and NIC for three classes of measurement errors though
 its NC is a little fewer than that of CoCoLasso, while NCL
 yields the highest relative RMSE and NIC. By comparing with
 the numerical comparison results in Section 5.1,
 the three solvers have similar performance as they do for
 those examples where the locations of the nonzero entries
 of $\beta^*$ are not fixed. That is, the correlation between
 the predictors has little influence on their performance.

 From the numerical comparisons in the last two subsections,
 when the true covariate matrix $X$ comes from the standard normal distribution
 (now our REC holds with a high probability) or other distributions such as
 the Laplace one in Example \ref{example5} and the exponential one
 in Example \ref{example7}, CaZnRLS is superior to CoCoLasso in terms of
 relative RMSE (especially for low noise cases) and NIC, although its NC is
 a little lower than that of CoCoLasso. As shown in Figure \ref{figure-time},
 CaZnRLS requires much less computing time.
 \begin{figure}[H]
 \setlength{\abovecaptionskip}{-0.2cm}
  \centerline{\epsfig{file=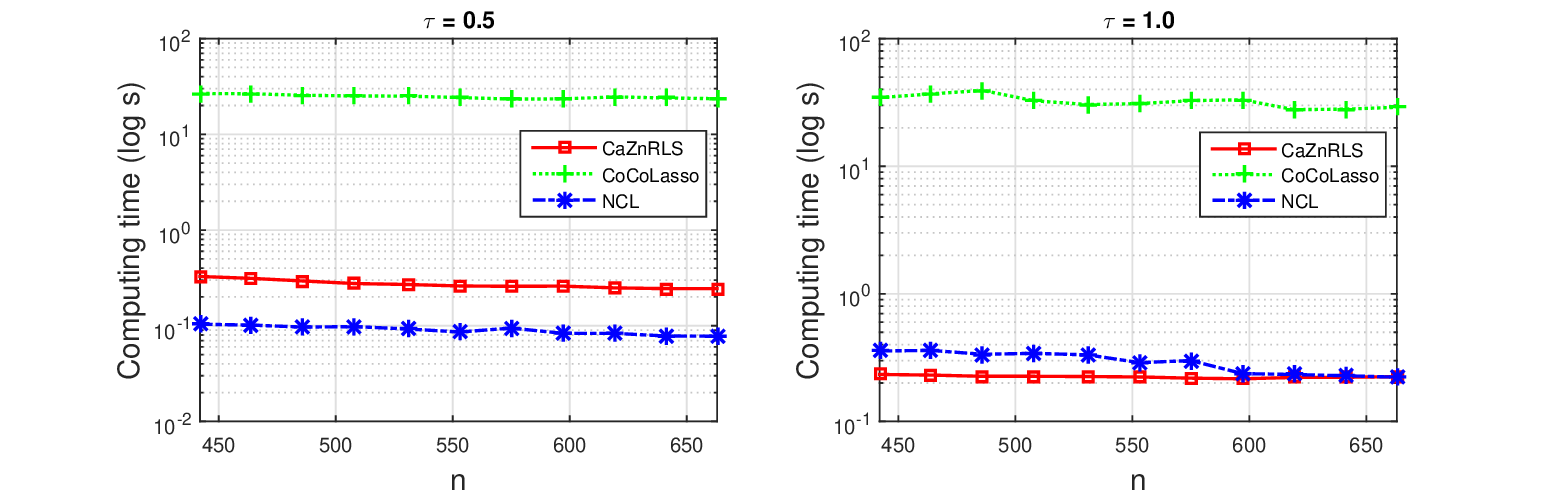,width=4.5in}}\par
 \caption{\small The computing time of three solvers for Example \ref{example2}}
 \label{figure-time}
 \end{figure}

%%%%%%%%%%%%%%%%%%%%%%%%%%%%%%%%%%%%%%%%%%%%%%%%%%%%%%%%%%%%%%%%%%%%%%%%%%%%%%%%%%%%%%%%%%%%%%%%%%%%%%%%%%%%%%%%%%%%%%%%%%%%
\vskip 14pt
\noindent {\large\bf Supplementary Materials}

 \noindent
 {\bf\large Appendix A}

 In this part, we write $\Delta\beta^k=\beta^k-\beta^{*}$ and $v^k=e-w^k$ for $k=1,2,\ldots$.

 \noindent
 {\bf A.1. The proof of Theorem \ref{errbound1}}

 To get the conclusion of Theorem \ref{errbound1}, we need
 the following two lemmas.
 %-------------------------------------------------------------------------Lemma2
 \begin{lemma}\label{relation}
  For any $\beta\in\mathbb{R}^p$, it holds that
  $\frac{1}{2n}\|\widetilde{Z}\beta\|^2\ge
  \frac{1}{2n}\|X\beta\|^2+\frac{1}{2}\beta^{\mathbb{T}}D\beta$.
 \end{lemma}
 \begin{proof}
 From $\widetilde{\Sigma}=\frac{1}{n}\widetilde{Z}^{\mathbb{T}}\widetilde{Z}$
 and $\widetilde{\Sigma}=\widehat{\epsilon}I+\Pi_{\mathbb{S}_{+}^p}(\widehat{\Sigma}-\widehat{\epsilon}I)$,
 for any $\beta\in\mathbb{R}^p$, we get
 \begin{align*}
  \frac{1}{2n}\big\|\widetilde{Z}\beta\big\|^2
  &=\frac{1}{2n}\|X\beta\|^2+\frac{1}{2}\beta^{\mathbb{T}}
    (\widetilde{\Sigma}-\widehat{\Sigma})\beta+\frac{1}{2}\beta^{\mathbb{T}}
    (\widehat{\Sigma}-\Sigma)\beta\\
 &=\frac{1}{2n}\|X\beta\|^2+\frac{1}{2}\beta^{\mathbb{T}}\Pi_{\mathbb{S}_{+}^p}(\widehat{\epsilon}I-\widehat{\Sigma})\beta
   +\frac{1}{2}\beta^{\mathbb{T}}D\beta\\
 &\ge \frac{1}{2n}\|X\beta\|^2 +\frac{1}{2}\beta^{\mathbb{T}}D\beta
 \end{align*}
 where the inequality is by the positive semidefiniteness of
 $\Pi_{\mathbb{S}_{+}^p}(\widehat{\epsilon}I-\widehat{\Sigma})$.
 \end{proof}
%---------------------------------------------------------------Lemma
 \begin{lemma}\label{err-lemma}
  Suppose that for some $k\ge 1$ there exists an index set $S^{k-1}\supseteq S^{*}$
  such that
  \(
    \max_{i\in(S^{k-1})^c}w^{k-1}_i\!\le \frac{1}{2}.
  \)
  Then, whenever $\lambda\ge 8\|\widetilde{\varepsilon}\|_\infty$, it holds that
  \begin{align*}
   \|\Delta\beta^k_{(S^{k-1})^c}\|_1\le 3\|\Delta\beta^k_{S^{k-1}}\|_1,\qquad\qquad\qquad\\
   \frac{1}{2n}\big\|\widetilde{Z}\Delta\beta^k\big\|^2
    \le\Big(\|\widetilde{\varepsilon}_{\!S^{k-1}}\|
    +\lambda\sqrt{{\textstyle\sum_{i\in S^{*}}}(v^{k-1}_i)^2}\Big)\big\|\Delta\beta_{\!S^{k-1}}^k\big\|.
  \end{align*}
 \end{lemma}
 \begin{proof}
  From the optimality of $\beta^k$ and the feasibility of $\beta^{*}$ to
  \eqref{expm-subx}, we have
 \[
  \frac{1}{2n}\big\|\widetilde{Z}\beta^k-\widetilde{y}\big\|^2
  +\lambda\sum_{i=1}^p v_i^{k-1}\big|\beta^k_{i}\big|
  \le\frac{1}{2n}\big\|\widetilde{Z}\beta^{*}-\widetilde{y}\big\|^2 +\lambda\sum_{i=1}^pv_i^{k-1}\big|\beta^{*}_{i}\big|
 \]
 which, by $\Delta\beta^k\!=\beta^k-\beta^{*}$ and
 $\widetilde{\varepsilon}\!=\frac{1}{n}\widetilde{Z}^{\mathbb{T}}(\widetilde{y}-\widetilde{Z}\beta^{*})$,
 can be rearranged as
 \begin{align}\label{noise1p-eb}
  \frac{1}{2n}\big\|\widetilde{Z}\Delta\beta^k\big\|^2
  &\le \langle\widetilde{\varepsilon},\Delta\beta^k\rangle
     +\lambda\!\sum_{i\in S^{*}} v_i^{k-1}\left(\big|\beta^{*}_{i}\big|-\big|\beta^k_{i}\big|\right)
     -\lambda\!\sum_{i\in (S^{*})^{c}} v_i^{k-1}\big|\beta^k_{i}\big|\nonumber\\
  &\le \langle\widetilde{\varepsilon},\Delta\beta^k\rangle
      +\lambda\!\sum_{i\in S^{*}} v^{k-1}_i\big|\Delta\beta^k_{i}\big|
      -\lambda\!\sum_{i\in(S^{k-1})^c}\!v^{k-1}_i\big|\Delta\beta^k_{i}\big|\\
  &\le \sum_{i\in S^{k-1}}\big|\widetilde{\varepsilon}_{i}\big|\big|\Delta\beta^k_{i}\big|
       +\sum_{i\in (S^{k-1})^{c}}\big|\widetilde{\varepsilon}_{i}\big|\big|\Delta\beta^k_{i}\big|\nonumber\\
  &\quad\ +\lambda\!\sum_{i\in S^{*}} v^{k-1}_i\big|\Delta\beta^k_{i}\big|
         -\lambda\!\sum_{i\in(S^{k-1})^c}\!v^{k-1}_i\big|\Delta\beta^k_{i}\big|\nonumber\\
  &\le (\lambda+\!\|\widetilde{\varepsilon}\|_\infty) \big\|\Delta\beta^k_{S^{k-1}}\big\|_1
      +\big(\|\widetilde{\varepsilon}\|_\infty\!-\!\lambda/2\big)
          \big\|\Delta\beta^k_{(S^{k-1})^c}\big\|_1\nonumber
 \end{align}
 where the second inequality is using $S^{k-1}\supseteq S^*$,
 and the last one is due to $v_i^k\le 1$ for $i\in S^*$ and
 $\min_{i\notin S^{k-1}}v^{k-1}_i\ge \frac{1}{2}$.
 From $\lambda\ge 8\|\widetilde{\varepsilon}\|_\infty$
 and $\frac{1}{2n}\big\|\widetilde{Z}\Delta\beta^k\big\|^2\ge 0$,
 we obtain the first inequality. For the second inequality,
 by using inequality \eqref{noise1p-eb}
 and $\min_{i\notin S^{k-1}}v^{k-1}_i\ge \frac{1}{2}$,
 it follows that
 \begin{align*}
   \frac{1}{2n}\big\|\widetilde{Z}\Delta\beta^k\big\|^2
   &\le \sum_{i=1}^{p}|\widetilde{\varepsilon}_{i}||\Delta\beta^k_{i}|
       -\frac{1}{2}\lambda\!\sum_{i\in(S^{k-1})^c}\!\big|\Delta\beta^k_{i}\big|
       +\lambda\!\sum_{i\in S^{*}}v^{k-1}_i\big|\Delta\beta^k_{i}\big|\nonumber\\
  &\le \sum_{i\in S^{k-1}}\big|\widetilde{\varepsilon}_{i}\big|\big|\Delta\beta^k_{i}\big|
        +\lambda\textstyle\sum_{i\in S^{*}}v^{k-1}_i\big|\Delta\beta^k_{i}\big|\\
  &\le \|\widetilde{\varepsilon}_{\!S^{k-1}}\|\big\|\Delta\beta_{\!S^{k-1}}^k\big\|
     +\lambda\!\sqrt{{\textstyle\sum_{i\in S^{*}}}(v^{k-1}_i)^2}\,\big\|\Delta\beta_{\!S^{k-1}}^k\big\|,
  \end{align*}
 where the second inequality is due to $\lambda\ge 8\|\widetilde{\varepsilon}\|_\infty$.
 \end{proof}

 \noindent
 {\bf The proof of Theorem \ref{errbound1}:}
 Define $S^{k-1}\!:= S^{*}\cup\{i\notin S^{*}\!: w_i^{k-1}>\frac{1}{2}\}$
 for each $k\in\mathbb{N}$.
 We first argue that if $|S^{l-1}|\leq 1.5s$ for some $l\in\mathbb{N}$,
 and consequently the following inequality holds
 \begin{equation}\label{betal}
   \big\|\Delta\beta^{l}\big\|\le \frac{2(\|\widetilde{\varepsilon}\|_{\infty}\sqrt{1.5s}
    +\lambda\sqrt{s})}{\kappa-24s\|D\|_{\rm max}}
    \le \frac{(2+\!\sqrt{6}/8)\lambda\!\sqrt{s}}{\kappa-24s\|D\|_{\rm max}}.
 \end{equation}
 Since $S^{l-1}\supseteq S^*$ with $|S^{l-1}|\leq 1.5s$ and
 $\lambda\ge 8\|\widetilde{\varepsilon}\|_\infty$,
 from Lemma \ref{err-lemma} we have
 \begin{align}
  \frac{1}{2n}\big\|\widetilde{Z}\Delta\beta^l\big\|^2
  &\le\left[\|\widetilde{\varepsilon}_{\!S^{l-1}}\|
    +\lambda\!\sqrt{{\textstyle\sum_{i\in S^{*}}}(v^{l-1}_i)^2}\right]
    \big\|\Delta\beta_{\!S^{l-1}}^l\big\|,\nonumber\\
  \label{Dterm}
  \big|(\Delta\beta^l)^{\mathbb{T}}D\Delta\beta^l\big|
  &\le\|D\|_{\rm max}\|\Delta\beta^l\|_{1}^{2}
  =\|D\|_{\rm max}\big(\|\Delta\beta^l_{ S^{l-1}}\|_{1}+\|\Delta\beta^l_{(S^{l-1})^{c}}\|_{1}\big)^{2}\nonumber\\
  &\le 16\|D\|_{\rm max}\|\Delta\beta^l_{ S^{l-1}}\|_{1}^2
  \le 16|S^{l-1}|\|D\|_{\rm max}\|\Delta\beta^l_{S^{l-1}}\|^2\nonumber\\
  &\le 24s\|D\|_{\rm max}\|\Delta\beta_{S^{l-1}}^l\|^2.
 \end{align}
 By combining the last two inequalities with Lemma \ref{relation},
 it then follows that
 \[
  \frac{1}{2n}\big\|X\Delta\beta^l\big\|^2-12s\|D\|_{\rm max}\|\Delta\beta^l\|^2
  \le\left[\|\widetilde{\varepsilon}_{\!S^{l-1}}\|
    +\lambda\!\sqrt{{\textstyle\sum_{i\in S^{*}}}(v^{l-1}_i)^2}\right]
    \big\|\Delta\beta_{\!S^{l-1}}^l\big\|.
 \]
 Notice that $\Delta\beta^l\in\mathcal{C}(S^*)$ since $S^{l-1}\supseteq S^*$
 with $|S^{l-1}|\leq 1.5s$. Together with the $\kappa$-REC of $\Sigma$ on $\mathcal{C}(S^*)$,
 it is immediate to obtain
 \begin{align}\label{betal-ineq}
 \!\frac{1}{2}\big(\kappa-24s\|D\|_{\rm max}\big)\|\Delta\beta^l\|^2
  &\le \left[\|\widetilde{\varepsilon}_{\!S^{l-1}}\|
    +\lambda\!\sqrt{{\textstyle\sum_{i\in S^{*}}}(v^{l-1}_i)^2}\right]
    \big\|\Delta\beta_{\!S^{l-1}}^l\big\|\\
  &\le\big[\|\widetilde{\varepsilon}\|_{\infty}\sqrt{|S^{l-1}|}
    +\lambda\sqrt{s}\big]\big\|\Delta\beta^l\big\|\nonumber\\
  &\le \big[\|\widetilde{\varepsilon}\|_{\infty}\sqrt{1.5s}
    +\lambda\sqrt{s}\big]\big\|\Delta\beta^l\big\|.\nonumber
 \end{align}
 This, by $\|\widetilde{\varepsilon}\|_\infty\le\frac{1}{8}\lambda$,
 implies that the inequality \eqref{betal} holds.

 Next we show that $|S^{k-1}|\!\le 1.5s$ for all $k\in\mathbb{N}$.
 When $k=1$, this inequality holds automatically since $S^0=S^{*}$
 implied by $w^0\le\frac{1}{2}e$. Now assume that $|S^{k-1}|\le1.5s$ for $k=l$
  with $l\ge 1$. From the above argument, we have
  \(
    \|\beta^{l}\!-\beta^{*}\|
    \le\frac{(2+\!\sqrt{6}/8)\lambda\!\sqrt{s}}{\kappa-24s\|D\|_{\rm max}}.
  \)
 Notice that $i\in S^l\backslash S^*$ implies $i\notin S^*$ and
 $w_i^{l}\in(\frac{1}{2},1]$. By equation \eqref{wik}, the latter
 implies $\rho_l|\beta^{l}_{i}|\ge1$. Consequently,
 \begin{align}\label{Lemma-noise42-eb}
  \sqrt{|S^l\backslash S^*|}
  &\le\sqrt{{\textstyle\sum_{i\in S^l\backslash S^*}}(\rho_l|\beta^{l}_{i}|)^2}
   \le \rho_l\|\beta^{l}-\beta^{*}\|\nonumber\\
  &\le \frac{(2+\!\sqrt{6}/8)\rho_{l}\lambda\!\sqrt{s}}{\kappa-24s\|D\|_{\rm max}}
  \le\sqrt{0.5s}
 \end{align}
 where the last inequality is by $\rho_{l}\lambda\leq\rho_{3}\lambda
 \le\frac{2(\kappa-24s\|D\|_{\rm max})}{5\sqrt{2}}$. Thus, $|S^{l}|\leq 1.5s$.
 Hence, $|S^{k-1}|\le 1.5s$ for all $k\in\mathbb{N}$,
 and the error bound follows from \eqref{betal}.

  \noindent
 {\bf A.2. The proof of Theorem \ref{errbound2}}

 To achieve the conclusion of Theorem \ref{errbound2}, we need
 the following lemma.
%--------------------------------------------------------------------------------
 \begin{lemma}\label{indexk}
  Let $F^k$ and $\Lambda^k$ be the sets in \eqref{FLambdak}.
  Then, for each $k\in\{0\}\cup\mathbb{N}$,
  \[
    \sqrt{{\textstyle\sum_{i\in S^{*}}}(v^{k}_i)^2}
    \le\sqrt{{\textstyle\sum_{i\in S^{*}}}\max(\mathbb{I}_{\Lambda^k}(i),\mathbb{I}_{F^k}(i))}.
  \]
 \end{lemma}
 \begin{proof}
  Fix an arbitrary $i\in S^*$. If $i\in F^k$,
  from $v^{k}_i=1-w_i^{k}\le 1$ we have $v^{k}_i\le\mathbb{I}_{F^k}(i)$.
  If $i\notin F^k$, from $v^{k}_i=1-w_i^{k}$ and \eqref{wik},
  it follows that
  \(
    v_i^k=\max\big(0,\min(1,\frac{2a-(a+1)\rho_k|\beta_i^k|}{2(a-1)})\big),
  \)
  and hence
  \(
   v_i^k\le\mathbb{I}_{\{i:\,\rho_k|\beta^{k}_{i}|\le 2a/(a+1)\}}(i)
   \le \mathbb{I}_{\Lambda^k}(i).
  \)
  Thus, for each $i$, it holds that
  \(
   (v^{k}_i)^2\le v^{k}_i\le\max(\mathbb{I}_{\Lambda^k}(i),\mathbb{I}_{F^k}(i)).
  \)
  From this, it is immediate to obtain the desired result.
 \end{proof}

 \noindent
 {\bf The proof of Theorem \ref{errbound2}:}
  Write
  \(
   S^{k-1}:=S^{*}\cup\{i\notin S^{*}\!: w_i^{k-1}>\frac{1}{2}\}
  \)
  for each $k\in\mathbb{N}$.
 Since the conclusion holds automatically for $k=1$, it suffices to
 consider the case $k\ge 2$. From the proof of Theorem \ref{errbound1},
 we know that $|S^{k-1}|\le 1.5s$ for all $k\in\mathbb{N}$. Moreover,
 by using \eqref{Lemma-noise42-eb} and $\rho_k\ge 1$,
 \begin{equation}\label{Geps-Sk}
   \big\|\widetilde{\varepsilon}_{\!S^{k-1}}\big\|
   \le \big\|\widetilde{\varepsilon}_{\!S^{*}}\big\|
        +\sqrt{|S^{k-1}\backslash S^{*}|}\big\|\widetilde{\varepsilon}\big\|_\infty
    \le\big\|\widetilde{\varepsilon}_{\!S^{*}}\big\|
       +\frac{\rho_{k-1}\lambda}{8}\sqrt{|S^{k-1}\backslash S^{*}|}.
 \end{equation}
  By using inequality \eqref{betal-ineq} and Lemma \ref{indexk}, it follows that
  \begin{align*}
   &\|\beta^k-\beta^{*}\|
     \le \frac{2}{\kappa-24s\|D\|_{\rm max}}\left[\|\widetilde{\varepsilon}_{\!S^{k-1}}\|
     +\lambda\!\sqrt{{\textstyle\sum_{i\in S^{*}}}(v^{k-1}_i)^2}\right]\\
   &\le \frac{2}{\kappa-\!24s\|D\|_{\rm max}}\left[\|\widetilde{\varepsilon}_{\!S^{k-1}}\|
     +\lambda\!\sqrt{{\textstyle\sum_{i\in S^{*}}}\max(\mathbb{I}_{\Lambda^{k-1}}(i),\mathbb{I}_{F^{k-1}}(i))}\right]\\
   &\le \frac{2}{\kappa-\!24s\|D\|_{\rm max}}\left[\|\widetilde{\varepsilon}_{\!S^{k-1}}\|
     +\lambda\!\sqrt{{\textstyle\sum_{i\in S^{*}}}\max\big(\mathbb{I}_{\Lambda^{k-1}}(i), \big| |\beta^{k-1}_{\!_i}|-|\beta^{*}_{i}|\big|^2(\rho_{k-1})^2\big)}\right]\\
   &\le \frac{2}{\kappa-\!24s\|D\|_{\rm max}}\Big(\|\widetilde{\varepsilon}_{\!S^{k-1}}\|
        +\lambda\sqrt{\max\big({\textstyle\sum_{i\in S^{*}}}\mathbb{I}_{\Lambda^{k-1}}(i),
        (\rho_{k-1})^2\|\Delta\beta^{k-1}\|^2\big)}\Big)
  \end{align*}
  where the third inequality is by the definition of $F^{k-1}$.
  Together with \eqref{Geps-Sk},
 \begin{align*}
  \|\beta^k-\beta^{*}\|
  &\le\frac{2}{\kappa-\!24s\|D\|_{\rm max}}\Big[\big\|\widetilde{\varepsilon}_{\!S^{*}}\big\|
    +\lambda\!\sqrt{{\textstyle\sum_{i\in S^{*}}}\mathbb{I}_{\Lambda^{k-1}}(i)}
    +\frac{9\rho_{k-1}\lambda}{8}\big\|\Delta\beta^{k-1}\big\|\Big]\nonumber\\
  %&\le\frac{2}{\kappa-\!24s\|D\|_{\infty}}\Big(\big\|\widetilde{\varepsilon}_{\!S^{*}}\big\|
%    +\lambda\!\sqrt{{\textstyle\sum_{i\in S^{*}}}\,\mathbb{I}_{\Lambda^{k-1}}(i)}\Big)
%    +\frac{9}{10\sqrt{2}}\|\beta^{k-1}-\beta^*\|\nonumber\\
 % &\le\frac{2}{\kappa-\!24s\|D\|_{\rm max}}\Big(\big\|\widetilde{\varepsilon}_{\!S^{*}}\big\|
%    +\lambda\!\sqrt{{\textstyle\sum_{i\in S^{*}}}\,\mathbb{I}_{\Lambda^{k-1}}(i)}\Big)
%    +\frac{1}{\sqrt{2}}\|\beta^{k-1}-\beta^*\|\\
    &\le\frac{2}{\kappa-\!24s\|D\|_{\rm max}}\Big(\big\|\widetilde{\varepsilon}_{\!S^{*}}\big\|
    +\lambda\!\sqrt{{\textstyle\sum_{i\in S^{*}}}\,\mathbb{I}_{\Lambda^{k-1}}(i)}\Big)
    +\frac{1}{\sqrt{2}}\|\beta^{k-1}-\beta^*\|
 \end{align*}
 where the second inequality is using
 $\rho_{k-1}\lambda\le\rho_{3}\lambda\leq\frac{2(\kappa-24s\|D\|_{\rm max})}{5\sqrt{2}}$.
 The desired result follows by solving this recursion
 with respect to $\|\beta^k-\beta^{*}\|$.

 \noindent
 {\bf A.3. The proof of Theorem \ref{sign-consistency}}

 We need the following two lemmas with
 $\Delta\widehat{\beta}^k=\beta^k-\beta^{\rm LS}$ for $k=1,2,\ldots$.
%-----------------------------------------------------------------------------------
 \begin{lemma}\label{signc-lemma1}
  Suppose that for some $k\ge 1$ there exists an index set $S^{k-1}\supseteq S^{*}$
  such that $\max_{i\in(S^{k-1})^c}w^{k-1}_i\!\le \frac{1}{2}$. Then,
  whenever $\lambda\ge 6\|\varepsilon^{{\rm LS}}\|_\infty$, it holds that
  \[
   \|\Delta\widehat{\beta}^k_{(S^{k-1})^c}\|_1\le 3\|\Delta\widehat{\beta}^k_{S^{k-1}}\|_1.
  \]
  \end{lemma}
  \begin{proof}
  By the optimality of $\beta^k$ and the feasibility of $\beta^{{\rm LS}}$
  to \eqref{expm-subx}, we have
  \[
   \frac{1}{2n}\big\|\widetilde{Z}\beta^k-\widetilde{y}\big\|^2
   +\lambda\sum_{i=1}^p v_i^{k-1}\big|\beta^k_{i}\big|
   \leq\frac{1}{2n}\big\|\widetilde{Z}\beta^{{\rm LS}}-\widetilde{y}\big\|^2
   +\lambda\sum_{i=1}^p v_i^{k-1}\big|\beta^{{\rm LS}}_{i}\big|,
  \]
  which, by $\Delta\widehat{\beta}^{k}=\beta^{k}-\beta^{\rm LS}$ and $\varepsilon^{{\rm LS}}=\frac{1}{n}\widetilde{Z}^{\mathbb{T}}(\widetilde{y}-\widetilde{Z}\beta^{\rm LS})$,
  can be rearranged as
 \begin{align*}
  \frac{1}{2n}\big\|\widetilde{Z}\Delta\widehat{\beta}^k\big\|^2
  &\le \langle \varepsilon^{{\rm LS}},\Delta\widehat{\beta}^k\rangle
       +\lambda\sum_{i=1}^p v_i^{k-1}(|\beta^{{\rm LS}}_{i}|-|\beta^k_{i}|)\\
  &=\sum_{i\notin S^{*}}\varepsilon_i^{{\rm LS}}\Delta\widehat{\beta}_i^k
        +\lambda\sum_{i\in S^{*}} v_i^{k-1}(\big|\beta^{{\rm LS}}_{i}\big|-\big|\beta^k_{i}\big|)
        -\lambda\sum_{i\notin S^{*}} v_i^{k-1}\big|\beta^k_{i}\big|\\
  &\le \sum_{i\notin S^{*}}\big|\varepsilon^{{\rm LS}}_{i}\big|\big|\Delta\widehat{\beta}^k_{i}\big|
      +\lambda\sum_{i\in S^{*}} v^{k-1}_i\big|\Delta\widehat{\beta}^{k}_{i}\big|
      -\lambda\sum_{i\notin S^{*}} v^{k-1}_i\big|\beta^{k}_{i}\big|
  \end{align*}
  where the equality is using $\varepsilon^{{\rm LS}}_{i}=0$ for $i\in S^*$
  and $\beta^{{\rm LS}}_{i}=0$ for all $i\notin S^{*}$. Now from
  $S^{k-1}\supseteq S^*$ and $v_i^{k-1}=1-w_i^{k-1}\ge 1/2$ for $i\notin S^{k-1}$,
  we obtain
  \begin{align}\label{ineq-sign1}
   \frac{1}{2n}\big\|\widetilde{Z}\Delta\widehat{\beta}^k\big\|^2
   &\le \sum_{i\notin S^{*}}\big|\varepsilon^{{\rm LS}}_{i}\big|\big|\Delta\widehat{\beta}_i^k\big|
      +\lambda\!\sum_{i\in S^{*}} v^{k-1}_i|\Delta\widehat{\beta}^k_{i}|
      -\lambda\!\sum_{i\notin S^{k-1}} v^{k-1}_i\big|\Delta\widehat{\beta}^k_{i}\big|\nonumber\\
   &\le \sum_{i\in S^{k-1}\backslash S^{*}}\big| \varepsilon^{{\rm LS}}_{i}\big|\big|\Delta\widehat{\beta}^k_{i}\big|
       +\lambda\!\sum_{i\in S^{*}} v^{k-1}_i\big|\Delta\widehat{\beta}^k_{i}\big|\nonumber\\
   &\qquad +\sum_{i\in(S^{k-1})^c}\big|\varepsilon^{{\rm LS}}_{i}\big|\big|\Delta\widehat{\beta}^k_{i}\big|
        -\frac{1}{2}\lambda\big\|\Delta\widehat{\beta}^k_{(S^{k-1)^{c}}}\big\|_1\\
   &\le \max\big(\|\varepsilon^{{\rm LS}}\|_\infty,\lambda\big)
        \big\|\Delta\widehat{\beta}_{S^{k-1}}^k\big\|_1
      +\!\big(\|\varepsilon^{{\rm LS}}\|_\infty\!-\frac{1}{2}\lambda\big)
         \big\|\Delta\widehat{\beta}^k_{(S^{k-1})^c}\big\|_1\nonumber
  \end{align}
  which along with the nonnegativity of $\frac{1}{2n}\|\widetilde{Z}\Delta\widehat{\beta}^k\|^2$
  implies the result.
  \end{proof}
 %-----------------------------------------------------------------------------------
 \begin{lemma}\label{signc-lemma2}
  Suppose that for some $k\ge 1$ there exists $S^{k-1}\supseteq S^{*}$
  with $|S^{k-1}|\le 1.5s$ such that $\max_{i\in(S^{k-1})^c}w^{k-1}_i\!\le\frac{1}{2}$,
  and that the matrix $\Sigma$ satisfies the $\kappa$-REC on $\mathcal{C}(S^*)$ with
  $\kappa>24s\|D\|_{\rm max}$. Then, when $\lambda\ge 6\|\varepsilon^{{\rm LS}}\|_\infty$,
  \[
    \big\|\Delta\widehat{\beta}^k\big\|
    \le \frac{2}{\kappa-24s\|D\|_{\rm max}}\Big(\big\|\varepsilon^{\rm{LS}}_{\!S^{k-1}}\big\|
     +\lambda\sqrt{{\textstyle\sum_{i\in S^{*}}}(v^{k-1}_i)^2}\Big).
  \]
 \end{lemma}
 \begin{proof}
  First of all, from equation \eqref{ineq-sign1} and
  $\lambda\ge 6\|\varepsilon^{{\rm LS}}\|_\infty$,
  it follows that
  \begin{align*}
   \frac{1}{2n}\big\|\widetilde{Z}\Delta\widehat{\beta}^k\big\|^2
   &\le \sum_{i\in S^{k-1}\backslash S^*}\big|\varepsilon^{{\rm LS}}_{i}\big|\big|\Delta\widehat{\beta}^k_{i}\big|
     +\lambda\sum_{i\in S^{*}} v^{k-1}_i\big|\Delta\widehat{\beta}^k_{i}\big|\\
   &\le \|\varepsilon_{S^{k-1}}^{{\rm LS}}\|\|\Delta\widehat{\beta}_{\!S^{k-1}}^k\|
       +\lambda\sqrt{{\textstyle\sum_{i\in S^{*}}}(v^{k-1}_i)^2}\,\|\Delta\widehat{\beta}_{S^{k-1}}^k\|
  \end{align*}
  where the second inequality is using $S^{k-1}\supseteq S^*$.
  Together with Lemma \ref{relation},
  \[
   \frac{1}{2n}\big\|X\Delta\widehat{\beta}^k\big\|^2
    \le\Big[\|\varepsilon_{S^{k-1}}^{{\rm LS}}\|
            +\lambda\sqrt{{\textstyle\sum_{i\in S^{*}}}(v^{k-1}_i)^2}\Big]
       \big\|\Delta\widehat{\beta}_{\!S^{k-1}}^k\big\|
     -\frac{1}{2}(\Delta\widehat{\beta}^k)^{\mathbb{T}}D\Delta\widehat{\beta}^k.
  \]
  Since $S^{k-1}\!\supseteq S^{*}$ with $|S^{k-1}|\!\le 1.5s$,
  using Lemma \ref{signc-lemma1} and the same arguments as for \eqref{Dterm}
  yields that $-(\Delta\widehat{\beta}^k)^{\mathbb{T}}D\Delta\widehat{\beta}^k
  \le 24s\|D\|_{\rm max}\|\Delta\widehat{\beta}^k\|^2$. Then,
  \[
    \frac{1}{2n}\big\|X\Delta\widehat{\beta}^k\big\|^2
    -12s\|D\|_{\rm max}\|\Delta\widehat{\beta}^k\|^2
    \le \Big[\|\varepsilon_{S^{k-1}}^{{\rm LS}}\|
            +\lambda\sqrt{{\textstyle\sum_{i\in S^{*}}}(v^{k-1}_i)^2}\Big]
       \big\|\Delta\widehat{\beta}_{\!S^{k-1}}^k\big\|.
  \]
  Since $\Sigma$ satisfies the $\kappa$-RSC on the set $\mathcal{C}(S^*)$
  with $\kappa>24s\|D\|_{\rm max}$, we have
  \[
    \frac{1}{2}(\kappa-24s\|D\|_{\rm max})\big\|\Delta\widehat{\beta}^k\big\|^2
    \le \Big[\|\varepsilon_{S^{k-1}}^{{\rm LS}}\|
            +\lambda\sqrt{{\textstyle\sum_{i\in S^{*}}}(v^{k-1}_i)^2}\Big]
       \big\|\Delta\widehat{\beta}_{\!S^{k-1}}^k\big\|.
  \]
  This implies the desired result. The proof is then completed.
 \end{proof}

 \noindent
 {\bf The proof of Theorem \ref{sign-consistency}:} Let
  \(
   S^{k-1}:=S^{*}\cup\{i\notin S^{*}\!: w_i^{k-1}>\frac{1}{2}\}
  \)
  for each $k\in\mathbb{N}$.
  We first prove that the desired inequalities holds
  by the induction on $k\in\mathbb{N}$.
  Since $w^0\le \frac{1}{2}e$, we have $S^0=S^{*}$ and $|S^0|=s$.
  Notice that $\Sigma$ satisfies the $\kappa$-REC on $\mathcal{C}(S^{*})$
  with $\kappa>24s\|D\|_{\rm max}$ and $\lambda\ge6\|\varepsilon^{\rm{LS}}\|_\infty$.
  The conditions of Lemma \ref{signc-lemma2} are satisfied.
  Along with $\varepsilon^{\rm{LS}}_{\!S^{*}}=0$ and $F^0=S^{*}$,
  \begin{align}\label{k1}
   \|\beta^1-\beta^{\rm{LS}}\|
   &\le \frac{2}{\gamma}\Big(\|\varepsilon^{\rm{LS}}_{\!S^{0}}\|
     +\!\lambda\sqrt{{\textstyle\sum_{i\in S^{*}}}(v^{0}_i)^2}\Big)\nonumber\\
   &\le \frac{2}{\gamma}\Big(\|\varepsilon^{\rm{LS}}_{\!S^{*}}\|+\!\lambda\sqrt{|F^0|}\Big)
   \le\frac{2.03\rho_0\lambda\sqrt{|F^0|}}{\gamma}.
  \end{align}
  Since $|\beta^{\rm{LS}}_{i}-\beta^{*}_{i}|\le \|\widetilde{\varepsilon}^\dagger\|_\infty$
  for $i\in S^*$ by \eqref{sign-ls} and $\rho_1\ge \gamma\lambda^{-1}\|\widetilde{\varepsilon}^\dag\|_\infty$,
  we have
 \[
  |\beta^{\rm{LS}}_{i}\!-\beta^1_{i}|
   \ge |\beta^{*}_{i}\!-\beta^1_{i}|-|\beta^{*}_{i}\!-\beta^{\rm{LS}}_{i}|
   \ge \frac{1}{\rho_1}-\frac{\rho_1\lambda}{\gamma}
   \ge\frac{9\sqrt{3}-4}{9\sqrt{3}\rho_1}\quad\forall i\in F^{1}
  \]
  where the last inequality is by $1\le\rho_1\le\sqrt{\frac{4\gamma}{9\sqrt{3}\lambda}}$.
  By the last two equations,
 \[
  \sqrt{|F^{1}|}
  =\sqrt{{\textstyle\sum_{i=1}^p}\mathbb{I}_{F^{1}}(i)}
  \le \frac{9\sqrt{3}\rho_1}{9\sqrt{3}\!-\!4}\sqrt{{\textstyle\sum_{i=1}^p}|\beta^{\rm{LS}}_{i}-\beta^1_{i}|^2}
  \le\frac{18.27\sqrt{3}\rho_1\rho_0\lambda}{(9\sqrt{3}\!-\!4)\gamma}\sqrt{|F^0|}.
 \]
 Together with \eqref{k1} and $1=\rho_0<\rho_1\le\rho_3$, we conclude that
 the desired inequalities holds for $k=1$.
 Now, assuming that the conclusion holds for $k\le l-1$ with $l\ge 2$,
 we prove that the conclusion holds for $k=l$. For this purpose,
 we first argue $|S^{l-1}|\leq 1.5s$. Indeed, for $i\in S^{l-1}\backslash S^*$,
 we have $w_i^{l-1}\in(\frac{1}{2},1]$, which by \eqref{wik} implies
 that $\rho_{l-1}|\beta^{l-1}_{i}|\ge 1$. Then,
 \begin{align*}
  \sqrt{|S^{l-1}\backslash S^*|}
  &\le\sqrt{|F^{l-1}|}\le \frac{18.27\sqrt{3}\rho_{l-1}\rho_{l-2}\lambda}
  {(9\sqrt{3}\!-\!4)\gamma}\sqrt{|F^{l-2}|}\leq\cdots\\
  &\le\Big(\frac{18.27\sqrt{3}\lambda}{(9\sqrt{3}\!-\!4)\gamma}\Big)^{l-1}\rho_{l-1}\rho_{l-2}^2
  \cdots\rho_2^2\rho_1\sqrt{|F^{0}|}\\
  &\le\sqrt{\Big(\frac{18.27\sqrt{3}(\rho_3)^2\lambda}{(9\sqrt{3}\!-\!4)\gamma}\Big)^{2l-2}|F^{0}|}
  \le\sqrt{\Big(\frac{8.12}{9\sqrt{3}\!-\!4}\Big)^{2l-2}|F^{0}|}\leq\sqrt{0.5s},
 \end{align*}
 where the first inequality is due to $S^{l-1}\backslash S^*\subseteq F^{l-1}$,
 the second is since the conclusion holds for $k\le l-1$ with $l\ge 2$,
 the next to the last is using $\rho_3\le\sqrt{\frac{4\gamma}{9\sqrt{3}\lambda}}$,
 and the last one is using $2l-2\ge 2$. The last inequality implies that
 $|S^{l-1}|\leq 1.5s$. Using Lemma \ref{signc-lemma2} delivers that
 \begin{align}
  \|\beta^l-\beta^{\rm{LS}}\|
  &\le\frac{2}{\gamma}\Big(\|\varepsilon^{\rm{LS}}_{S^{l-1}}\|
            +\lambda\sqrt{{\textstyle\sum_{i\in S^{*}}\,(v^{l-1}_i)^2}}\Big)\nonumber\\
  &\le\frac{2}{\gamma}\Big(\|\varepsilon^{\rm{LS}}_{S^{l-1}\backslash S^{*}}\|
      +\lambda\sqrt{{\textstyle\sum_{i\in S^{*}}\mathbb{I}_{F^{l-1}}(i)}}\Big)\nonumber\\
  &\le \frac{2}{\gamma}\Big(\|\varepsilon^{\rm{LS}}\|_\infty\sqrt{|S^{l-1}\backslash S^{*}|}
       +\lambda\sqrt{|F^{l-1}\cap S^{*}}|\Big)\nonumber\\
  &\le\frac{2\lambda}{\gamma}\Big(\frac{1}{6}\sqrt{|F^{l-1}\backslash S^{*}|}
       +\sqrt{|F^{l-1}\cap S^{*}|}\Big)\nonumber\\
  &\leq\frac{2\lambda}{\gamma}\sqrt{(1+\!1/36)|F^{l-1}|}
   \le\frac{2.03\rho_{l-1}\lambda}{\gamma}\sqrt{|F^{l-1}|},\nonumber
  \end{align}
  where the second inequality is using $\varepsilon^{\rm{LS}}_{S^{*}}=0$,
  Lemma \ref{indexk} and $\rho_{l-1}\ge\rho_1>\frac{4a}{(a+1)\min_{i\in S^{*}}|\beta_{i}|}$,
  the fourth one is due to $\lambda\ge 6\|\varepsilon^{\rm{LS}}\|_\infty$,
  and the fifth one is since $\frac{1}{6}a+b\leq\sqrt{(1+\!\frac{1}{36})(a^2+b^2)}$
  for all $a,b\in \mathbb{R}$. Now using the same argument as those for $k=1$,
  we have
  \(
   |\beta^l_{i}-\beta^{\rm{LS}}_{i}|
   \ge \frac{9\sqrt{3}-4}{9\sqrt{3}\rho_l}
  \)
  for all $i\in F^{l}$, and hence
  \(
   \sqrt{|F^{l}|}\le\frac{18.27\sqrt{3}\rho_l\rho_{l-1}\lambda}{(9\sqrt{3}\!-\!4)\gamma}\sqrt{|F^{l-1}|}.
  \)
  Thus, we complete the proof of the case $k=l$,
  and the desired inequalities hold for all $k$.

  Note that $(\rho_3)^2\lambda\le \frac{4\gamma}{9\sqrt{3}}$ and
  $\rho_k\le \rho_3$ for all $k\in\mathbb{N}$. So, it holds that
  \[
   \sqrt{|F^{\overline{k}}|}\le
    \frac{18.27\sqrt{3}\rho_{\overline{k}}\rho_{\overline{k}-1}\lambda}{(9\sqrt{3}\!-\!4)\gamma}\sqrt{|F^{\overline{k}-1}|}
    \leq\cdots\le\Big(\frac{18.27\sqrt{3}(\rho_3)^2\lambda}{(9\sqrt{3}\!-\!4)\gamma}\Big)^{\overline{k}}\sqrt{|F^{0}|}<1,
  \]
  which implies that $|F^{k}|=0$ when $k\ge \overline{k}$.
  Together with the first inequality obtained, we have
  $\beta^k=\beta^{\rm{LS}}$ when $k\geq \overline{k}$.
  From $\rho_3\le\sqrt{\frac{4\gamma}{9\sqrt{3}\lambda}}$ and \eqref{sign-ls},
  \begin{equation}\label{final-ineq}
   \big||\beta_{i}^*|-|\beta^{\rm LS}_{i}|\big|
   \le|\beta^{*}_{i}-\beta^{\rm LS}_{i}|
   \le\|\widetilde{\varepsilon}^\dagger\|_\infty\le\rho_k\lambda\gamma^{-1}
   \le\frac{4}{9\sqrt{3}\rho_k}\quad\forall i\in S^{*}.
  \end{equation}
  This, along with $\min_{i\in S^{*}}\!|\beta^*_{i}|\ge\frac{4a}{(a+1)\rho_k}>\frac{4}{9\sqrt{3}\rho_k}$,
  implies $|\beta^{\rm{LS}}_{i}|>0$ for all $i\in S^{*}$ (if not,
  one will obtain $\frac{a}{a+1}\le\frac{1}{9\sqrt{3}}$, a contradiction
  to $a>1$), and hence ${\rm supp}(\beta^{\rm{LS}})=S^{*}$.
  The last inequality also implies
  ${\rm sign}(\beta^{\rm{LS}})={\rm sign}(\beta^{*})$ (if not,
  there exists $i_0\in S^*$ such that
  ${\rm sign}(\beta^{\rm LS}_{i_0})=-{\rm sign}(\beta^{*}_{i_0})$
  and then $|\beta^{*}_{i_0}-\beta^{\rm LS}_{i_0}|>|\beta^{*}_{i_0}|\geq
  \min_{i\in S^{*}}\!|\beta^*_{i}|>\frac{4}{9\sqrt{3}\rho_k}$,
  a contradiction to \eqref{final-ineq}.)
  Thus, $\beta^k=\beta^{\rm{LS}}$ and ${\rm sign}(\beta^k)={\rm sign}(\beta^{*})$
  for all $k\ge\overline{k}$. We complete the proof.

 \noindent
 {\bf\large Appendix B}

 In this part, we need the following assumption on the noise vector $\varepsilon$.
 \begin{assumption}\label{noise-assump}
  Assume that $\varepsilon_i\,(i=1,\ldots,m)$ are i.i.d. sub-Gaussians,
  i.e., there is $\sigma\!>\!0$ such that
  \(
   \mathbb{E}[\exp(t\varepsilon_i)]\leq\exp(\sigma^2t^2/2)
  \)
  for all $i$ and $t\!\in\! \mathbb{R}$.
 \end{assumption}
%--------------------------------------------------------------------------
 \noindent {\bf B.1. Additive errors case}\label{sec4.3.1}

  In this part, we consider that the matrix $X$ is contaminated by
  additive measurement errors, i.e., $Z=X+A$, where $A=(a_{ij})$
  is the matrix of measurement errors and the rows of $A$
  are assumed to be i.i.d. with zero mean, finite covariance $\Sigma_A$
  and sub-Gaussian parameter $\tau^2$. Following the line of \cite{Loh14},
  we assume that $\Sigma_A$ is known. Now the unbiased surrogates of
  $\Sigma$ and $\xi$ are given by
  $\widehat{\Sigma}_{\rm add} =\frac{1}{n}Z^{\mathbb{T}}Z-\Sigma_A$
  and $\widehat{\xi}_{\rm add}=\frac{1}{n}Z^{\mathbb{T}}y$, respectively.
  We write
  \(
   \widetilde{\Sigma}_{\rm add}:=\widehat{\epsilon} I
   +\Pi_{\mathbb{S}_{+}^p}(\widehat{\Sigma}_{\rm add}-\widehat{\epsilon}I)
  \)
  and
  \(
   \widetilde{\varepsilon}_{\rm add}:=\widehat{\xi}_{\rm add}-\widetilde{\Sigma}_{\rm add}\beta^{*}.
 \)
%----------------------------------------------------------------------------
 \begin{lemma}\label{addnoise}
  Let $K:=2(\lambda_{\rm max}(\Sigma_A)+\widehat{\epsilon})\|\beta^*\|_1$
  and $\eta=\min\big(1,\frac{\epsilon_0}{\lambda_{\rm max}(\Sigma_A)+\widehat{\epsilon}}\big)$.
  Then, there exist universal positive constants $C$ and $c$, and positive function
  $\widehat{\zeta}$ (depending only on $\beta^*,\tau^2$, $\sigma^2$ and $\lambda_{\rm max}(\Sigma_A)$)
  such that
  \begin{align}{}\label{additive-ineq1}
    \mathbb{P}\{\|(\widetilde{\Sigma}_{\rm add}-\Sigma)\beta^*\|_\infty>K\}
    \le Cp^2\exp(-cn\widehat{\zeta}^{-1}\eta^2),\\
    \mathbb{P}\{\|\widetilde{\varepsilon}_{\rm add}\|_\infty>K\}
    \le Cp^2\exp(-cns^{-2}\widehat{\zeta}^{-1}\eta^2).\qquad
    \label{additive-ineq2}
  \end{align}
 \end{lemma}
 \begin{proof}
  From the expression of $\widetilde{\Sigma}_{\rm add}$,
  it follows that
  \begin{align*}
    \|(\widetilde{\Sigma}_{\rm add}-\Sigma)\beta^*\|_{\infty}
   &\le\|(\widetilde{\Sigma}_{\rm add}-\widehat{\Sigma}_{\rm add})\beta^*\|_{\infty}
     +\|(\widehat{\Sigma}_{\rm add}-\Sigma)\beta^*\|_\infty\\
   &=\|\Pi_{\mathbb{S}_{+}^p}(\widehat{\epsilon} I\!-\!\widehat{\Sigma}_{\rm add})\beta^*\|_\infty
      +\|(\widehat{\Sigma}_{\rm add}-\Sigma)\beta^*\|_\infty\\
   &\le \|\Pi_{\mathbb{S}_{+}^p}(\widehat{\epsilon}I\!-\!\widehat{\Sigma}_{\rm add})\|_{\rm max}\|\beta^*\|_1
       +\|(\widehat{\Sigma}_{\rm add}-\Sigma)\beta^*\|_\infty.
  \end{align*}
  For a matrix $\Gamma\in\mathbb{S}_{+}^{p}$,
  it is not hard to check that
  \(
    \lambda_{\rm max}(\Gamma)\ge \|\Gamma\|_{\rm max}.
  \)
  Thus,
  \begin{align}\label{temp-ineq41}
    \|(\widetilde{\Sigma}_{\rm add}-\Sigma)\beta^*\|_\infty
   &\le \lambda_{\rm max}\big[\Pi_{\mathbb{S}_{+}^p}(\widehat{\epsilon}I\!-\!\widehat{\Sigma}_{\rm add})\big]\|\beta^*\|_1
      + \|(\widehat{\Sigma}_{\rm add}-\Sigma)\beta^*\|_\infty\nonumber\\
   &= \big[\widehat{\epsilon}-\lambda_{\rm min}(\widehat{\Sigma}_{\rm add})\big]\|\beta^*\|_1
      +\|(\widehat{\Sigma}_{\rm add}-\Sigma)\beta^*\|_\infty.
  \end{align}
  Notice that
  \(
   \lambda_{\rm min}(\widehat{\Sigma}_{\rm add})\ge\lambda_{\rm min}(\frac{1}{n}Z^{\mathbb{T}}Z)
  -\!\lambda_{\rm max}(\Sigma_A)\ge-\lambda_{\rm max}(\Sigma_A)
  \)
  implied by Theorem 4.3.7 of \cite{RH90}.
  Together with \eqref{temp-ineq41},
  \begin{equation*}
   \|(\widetilde{\Sigma}_{\rm add}-\Sigma)\beta^*\|_\infty
   \le (\widehat{\epsilon}+\!\lambda_{\rm max}(\Sigma_A))\|\beta^*\|_1
        +\|(\widehat{\Sigma}_{\rm add}-\Sigma)\beta^*\|_\infty.
  \end{equation*}
  By this and Lemma 1 of \cite{Datta16} with
  $\epsilon=\frac{K\eta}{2\|\beta^*\|_1}\leq \epsilon_0$,
  there exist universal positive constants $C,c$ and positive functions $\zeta$
  (depending only on $\beta^*,\tau^2$, $\sigma^2$ and $\lambda_{\rm max}(\Sigma_A)$) such that
  \begin{align*}
    \mathbb{P}\{\|(\widetilde{\Sigma}_{\rm add}-\Sigma)\beta^*\|_\infty>K\}
    &\le\mathbb{P}\Big\{\|(\widehat{\Sigma}_{\rm add}-\Sigma)\beta^*\|_\infty>K/2\Big\}\\
    &\le\mathbb{P}\Big\{\|\widehat{\Sigma}_{\rm add}-\Sigma\|_{\rm max}>\frac{K\eta}{2\|\beta^*\|_1}\Big\}\\
    &\le Cp^2\exp(-c n\eta^2(\lambda_{\rm max}(\Sigma_A)+\widehat{\epsilon})^2\zeta^{-1}).
  \end{align*}
  This shows that \eqref{additive-ineq1} holds. Recall that
  $\widetilde{\varepsilon}_{\rm add}=\widehat{\xi}_{\rm add}-\widetilde{\Sigma}_{\rm add}\beta^{*}$.
  Hence,
  \[
    \|\widetilde{\varepsilon}_{\rm add}\|_\infty
     \le\|\widehat{\xi}_{\rm add}-\xi\|_\infty+\|\xi-\Sigma\beta^{*}\|_\infty
      +\|(\widetilde{\Sigma}_{\rm add}-\Sigma)\beta^*\|_\infty.
  \]
  By applying Lemma 1 of \cite{Datta16} with $\epsilon=\frac{K\eta_0}{3}\leq\epsilon_0$
  where $\eta_0=\min(1,\frac{1.5\eta}{\|\beta^*\|_1})$, we obtain
  \[
    \mathbb{P}\Big\{\|\widehat{\xi}_{\rm add}-\xi\|_\infty\ge \frac{K}{3}\Big\}
    \le \mathbb{P}\Big\{\|\widehat{\xi}_{\rm add}-\xi\|_\infty\ge \frac{K\eta_0}{3}\Big\}
    \le Cp\exp(-ncs^{-2}K^2\eta_0^2\zeta^{-1}),
  \]
  while $\mathbb{P}\{\|\xi-\Sigma\beta^{*}\|_\infty\ge K/3\}\le Cp\exp(-nc\sigma^{-2}K^2)$
  holds by Property B.2 of \cite{Datta16}.
  Together with the last inequality and inequality \eqref{additive-ineq1},
  we obtain the inequality \eqref{additive-ineq2}.
  \end{proof}

 Lemma \ref{addnoise} states that $\|(\widetilde{\Sigma}_{\rm add}-\Sigma)\beta^*\|_{\infty}$
 and $\|\widehat{\xi}_{\rm add}\|_{\infty}$ can be controlled by $\|\beta^*\|_1$.
  From the proof of Theorem 1 in \cite{Datta16}, we know that there also exist universal
  positive constants $C'$ and $c'$ and positive function $\widehat{\zeta}'$
  (depending on $\beta_{S^*}^*,\tau^2$ and $\sigma^2$)
  such that for all $\epsilon\le\min(\epsilon_0,\frac{\kappa}{64s})$,
  \begin{equation}\label{Zou-ineq}
    \mathbb{P}\big\{\|D\|_{\rm max}\ge {\kappa}/{(64s)}\big\}
    \le C'p^2\exp(-nc'\epsilon^2(\widehat{\zeta}')^{-1}).
  \end{equation}
  Combining with Lemma \ref{addnoise} and Theorem \ref{errbound2},
  we have the following result.
%-----------------------------------------------------------------Corollary
 \begin{corollary}\label{corollary-errbound}
  Suppose that $\Sigma$ satisfies the $\kappa$-REC on $\mathcal{C}(S^{*})$.
  If $\lambda$ and $\rho_3$ in Algorithm \ref{Alg}
  are chosen such that $\lambda\ge 8K$ and $\rho_3\le\frac{\kappa}{4\sqrt{2}\lambda}$
  where $K$ is the constant same as in Lemma \ref{addnoise},
  then for all $k\in\mathbb{N}$ the following inequality
 \begin{equation}\label{err-ineq1}
  \|\beta^k-\beta^{*}\|\le\frac{4\sqrt{s}\,\lambda}{\kappa}
 \end{equation}
 holds w.p. at least $1-p^2 C\exp(-cns^{-2}\zeta^{-1})$,
 where $C$ and $c$ are universal positive constants and $\zeta$
 is a positive function on $\beta^*,\tau^2,\sigma^2,\kappa$
 and $\lambda_{\rm max}(\Sigma_A)$.
 \end{corollary}

 Write $\widetilde{G}_{\rm add}:=[\widetilde{\Sigma}_{\rm add}]_{(S^*)^cS^*}
 [\widetilde{\Sigma}_{\rm add}]_{S^*S^*}^{-1}$. By recalling
 $\varepsilon^{{\rm LS}}=\frac{1}{n}\widetilde{Z}^{\mathbb{T}}(\widetilde{y}-\widetilde{Z}\beta^{\rm LS})$
 and using the equality \eqref{sign-ls}, it is not difficult to obtain
 the inequalities
 \[
  \|\varepsilon^{\rm LS}\|_{\infty}
  \le\!\max(2,1\!+s\|\widetilde{G}_{\rm add}\|_{\rm max})
  \|\widetilde{\varepsilon}_{\rm add}\|_\infty,\,
  \|\widetilde{\varepsilon}^\dag\|_\infty
  \!\le s\|[\widetilde{\Sigma}_{\rm add}]_{\!S^*S^*}^{-1}\|_{\rm max}
  \|\widetilde{\varepsilon}_{\rm add}\|_\infty.
 \]
 Along with Lemma \ref{addnoise}, Theorem \ref{sign-consistency} and
 \eqref{Zou-ineq}, we obtain the following result.
 %-----------------------------------------------------------------Corollary
 \begin{corollary}\label{corollary-sign}
  Suppose that $\Sigma$ satisfies the $\kappa$-REC on the set $\mathcal{C}(S^{*})$.
  Write $K'=K\max(2,1\!+\!s\|\widetilde{G}_{\rm add}\|_{\rm max})$
  and $K''=Ks\|[\widetilde{\Sigma}_{\rm add}]_{\!S^*S^*}^{-1}\|_{\rm max}$
  where the constant $K$ is same as the one in Lemma \ref{addnoise}.
  If $\lambda,\rho_1$ and $\rho_3$ are chosen such that
  $\lambda\ge 6K'$, $\rho_1\!>\max\!\big(\frac{4a}{(a+1)\min_{i\in S^*}\!|\beta^{*}_{i}\!|},
  \frac{5\kappa K''}{8\lambda})$ and $\rho_3\le\!\sqrt{\frac{5\kappa}{18\sqrt{3}\lambda}}$,
  then $\beta^k=\beta^{\rm{LS}}$ and ${\rm sign}(\beta^{k})={\rm sign}(\beta^{*})$
  for $k\ge\widehat{k}=\lceil\frac{0.5\ln(s)}{\ln[(9\sqrt{3}\!-\!4)5\kappa\lambda^{-1}]
  -\ln[147\sqrt{3}(\rho_3)^2]}\rceil$ w.p. at least $1-Cp^2\exp(-cns^{-2}\zeta^{-1})$,
  where $C,c$ are universal positive constants and $\zeta$ is a positive
  function depending on $\beta^*,\tau^2,\sigma^2,\kappa$ and $\lambda_{\rm max}(\Sigma_A)$.
 \end{corollary}

 As remarked in the beginning of this subsection,
 when $X$ is from the $\Sigma_x$-Gaussian ensemble, with high probability
 there exists a constant $\kappa>0$ such that $\Sigma$ satisfies
 the REC on $\mathcal{C}(S^*)$. We see that if $\kappa$ has a small value,
 there is a great possibility for the choice range of $\rho_3$ to be empty,
 and it is impossible to achieve the sign consistency; and when $\kappa$
 is not too small, say, $\frac{5\kappa}{108\sqrt{3}K'}>1$, after
 $k\ge \widehat{k}\ge\lceil\frac{0.5\ln(s)}{\ln(1.42)}\rceil$
 the iterate $\beta^k$ is sign-consistent.

%------------------------------------------------------------------------------
 \noindent {\bf B.2. Multiplicative errors and missing data}\label{sec4.3.2}

 In this part, we consider that the matrix $X$ is contaminated by
 multiplicative measurement errors, i.e. $Z=X\circ M$, where $M=(m_{ij})$
 is the matrix of measurement errors and the rows of $M$ are assumed to be
 i.i.d. with mean $\mu_M$, covariance $\Sigma_M$ and sub-Gaussian parameter $\tau^2$.
 Similar to \cite{Datta16}, in the sequel we need the following conditions
 \begin{equation}\label{const}
  \max_{i,j}|X_{ij}|\le c_{X},\,
  \max_{i,j}|M_{ij}|\le c_M,\,\min_{i,j}(\Sigma_M)_{ij}>0,\,
  (\mu_M)_{\rm min}>0
 \end{equation}
 where $c_X$ and $c_M$ are universal positive constants. From \cite{Loh11},
 \(
    \widehat{\Sigma}_{\rm mul}=\frac{1}{n}Z^{\mathbb{T}}Z\oslash (\Sigma_M+\mu_M\mu_M^{\mathbb{T}})
 \)
 and $\widehat{\xi}_{\rm mul}=\frac{1}{n}Z^{\mathbb{T}}y\oslash\mu_M$
 are the unbiased surrogates of $\Sigma$ and $\xi$,
 where $\oslash$ denotes the elementwise division operator.
 Let $\widetilde{\Sigma}_{\rm mul}:=\widehat{\epsilon} I
   +\Pi_{\mathbb{S}_{+}^p}(\widehat{\Sigma}_{\rm mul}\!-\widehat{\epsilon}I)$
 and $\widetilde{\varepsilon}_{\rm mul}:=\widehat{\xi}_{\rm mul}-\widetilde{\Sigma}_{\rm mul}\beta^{*}$.
%---------------------------------------------------------------------------------------------------
 \begin{lemma}\label{mulnoise}
  Let $\widetilde{K}\!:=2\big[\widehat{\epsilon}-\min(\lambda_{\rm min}(\Sigma_M^\dag),0)
 c_{M}^2\big]\|\beta^*\|_1$ with $\Sigma_M^\dag=E\oslash(\Sigma_M+\mu_M\mu_M^{\mathbb{T}})$
 where $E$ is the matrix of all ones and
 $\widetilde{\eta}=\min\big(1,\frac{\epsilon_0}{\widehat{\epsilon}-\min(\lambda_{\rm min}(\Sigma_M^\dag),0)
 c_{M}^2}\big)$. Then, there exist universal positive constants
  $\widetilde{C},\widetilde{c}$ and positive function $\widetilde{\zeta}$
  (depending on $\beta^*,\tau^2,\sigma^2,\lambda_{\rm min}(\Sigma_M^\dag)$ and
  the constants in \eqref{const}) such that
  \begin{align}\label{Multi-ineq1}
    \mathbb{P}\{\|(\widetilde{\Sigma}_{\rm mul}-\Sigma)\beta^*\|_\infty>\widetilde{K}\}
    \le \widetilde{C}p^2\exp(-\widetilde{c}n\widetilde{\zeta}^{-1}\widetilde{\eta}^2),\\
    \mathbb{P}\{\|\widetilde{\varepsilon}_{\rm mul}\|_\infty>\widetilde{K}\}
    \le \widetilde{C}p^2\exp(-\widetilde{c}ns^{-2}\widetilde{\zeta}^{-1}\widetilde{\eta}^2).\qquad
    \label{Multi-ineq2}
  \end{align}
 \end{lemma}
 \begin{proof}
  From the expression of $\widetilde{\Sigma}_{\rm mul}$
  and the proof of Lemma \ref{addnoise}, we have
  \begin{equation}\label{temp-mul}
   \|(\widetilde{\Sigma}_{\rm mul}-\Sigma)\beta^*\|_\infty
   \le\big[\widehat{\epsilon}-\lambda_{\rm min}(\widehat{\Sigma}_{\rm mul})\big]\|\beta^*\|_1
      +\|(\widehat{\Sigma}_{\rm mul}-\Sigma)\beta^*\|_\infty.
 \end{equation}
  Next we provide a lower bound for $\lambda_{\rm min}(\widehat{\Sigma}_{\rm mul})$.
  Write $\Sigma_Z=\frac{1}{n}Z^{\mathbb{T}}Z$. Then,
  \begin{align*}
  \lambda_{\rm min}(\widehat{\Sigma}_{\rm mul})
   &=\lambda_{\rm min}\big[\Sigma_Z\circ(\Sigma_M^\dag-\lambda_{\rm min}(\Sigma_M^\dag)I)+(\Sigma_Z\circ\lambda_{\rm min}(\Sigma_M^\dag)I)\big]\\
   &\ge\lambda_{\rm min}\big[\Sigma_Z\circ(\Sigma_M^\dag-\lambda_{\rm min}(\Sigma_M^\dag)I)\big]
       +\lambda_{\rm min}[\Sigma_Z\circ\lambda_{\rm min}(\Sigma_M^\dag)I]\\
   &\ge \lambda_{\rm min}(\Sigma_Z)\lambda_{\rm min}(\Sigma_M^\dag-\lambda_{\rm min}(\Sigma_M^\dag)I)
        +\lambda_{\rm min}[\Sigma_Z\circ\lambda_{\rm min}(\Sigma_M^\dag)I]\\
   &\ge \lambda_{\rm min}[\Sigma_Z\circ\lambda_{\rm min}(\Sigma_M^\dag)I]
   \geq \min(\lambda_{\rm min}(\Sigma_M^\dag),0)\max_{1\le j\le p}(Z_j^{\mathbb{T}}Z_j/n)\\
   &\ge \min(\lambda_{\rm min}(\Sigma_M^\dag),0)c_{M}^2
  \end{align*}
  where the first inequality is using Theorem 4.3.1 of \cite{RH90},
  the second one is due to $\Sigma_M^\dag-\lambda_{\rm min}(\Sigma_M^\dag)I\succeq 0$
  and Theorem 5.3.1 of \cite{RH91}, the fourth one is using
  the positive semidefiniteness of $\Sigma_Z$, and the last one
  is due to $Z=X\circ M$ and the first two relations in \eqref{const}.
  Together with \eqref{temp-mul} and the definition of $\widetilde{K}$,
  \begin{equation*}
   \|(\widetilde{\Sigma}_{\rm mul}-\Sigma)\beta^*\|_\infty
   \le (\widetilde{K}/2)+\|(\widehat{\Sigma}_{\rm mul}-\Sigma)\beta^*\|_\infty.
  \end{equation*}
  By Lemma 2 of \cite{Datta16} for
  $\epsilon=\frac{\widetilde{K}\widetilde{\eta}}{2\|\beta^*\|_1}\leq \epsilon_0$,
  there are universal positive constants $C,c$ and positive functions
  $\zeta$ (depending on $\beta^*,\tau^2,\sigma^2$) and the constants
  in \eqref{const} such that
 \begin{align*}
  &\mathbb{P}\{\|(\widetilde{\Sigma}_{\rm mul}\!-\Sigma)\beta^*\|_\infty>\widetilde{K}\}
  \le\mathbb{P}\Big\{\|(\widehat{\Sigma}_{\rm mul}\!-\Sigma)\beta^*\|_\infty>\!\frac{\widetilde{K}}{2}\Big\}\\
  &\le\mathbb{P}\Big\{\|(\widehat{\Sigma}_{\rm mul}\!-\Sigma)\beta^*\|_\infty>\!\frac{\widetilde{K}\widetilde{\eta}}{2}\Big\}
  \le\mathbb{P}\Big\{\|\widehat{\Sigma}_{\rm mul}\!-\Sigma\|_{\rm max} >\!\frac{\widetilde{K}\widetilde{\eta}}{2\|\beta^*\|_1}\Big\}\\
  &\le Cp^2\exp\big(-c n(\widehat{\epsilon}-\min(\lambda_{\rm min}(\Sigma_M^\dag),0)
 c_{M}^2)^2\widetilde{\eta}^2\zeta^{-1}\big).
 \end{align*}
 Thus, we get \eqref{Multi-ineq1}. From Property B.2 of \cite{Datta16}
 and
 \(
   \|\widetilde{\varepsilon}_{\rm mul}\|_\infty
    \le\|\widehat{\xi}_{\rm mul}-\xi\|_\infty+\|\xi-\Sigma\beta^{*}\|_\infty
      +\|(\widetilde{\Sigma}_{\rm mul}-\Sigma)\beta^*\|_\infty,
 \)
 it follows that
  $\mathbb{P}\{\|\xi-\Sigma\beta^{*}\|_\infty\ge\widetilde{K}/3\}
  \le Cp\exp(-nc\sigma^{-2}\widetilde{K}^2)$.
  Together with Lemma 2 of \cite{Datta16} and the inequality
  \eqref{Multi-ineq1}, we obtain \eqref{Multi-ineq2}.
 \end{proof}

  By using Lemma \ref{mulnoise} and the same arguments as those for
  Corollary \ref{corollary-errbound} and \ref{corollary-sign},
  the following conclusions hold
  where $\widetilde{G}_{\rm mul}:=[\widetilde{\Sigma}_{\rm mul}]_{(S^*)^cS^*}
 [\widetilde{\Sigma}_{\rm mul}]_{S^*S^*}^{-1}$.
%---------------------------------------------------------------------------------------------------------
 \begin{corollary}\label{corollary-errbound-mul}
  Suppose that $\Sigma$ satisfies the $\kappa$-REC on the set $\mathcal{C}(S^{*})$.
  If $\lambda$ and $\rho_3$ are chosen such that $\lambda\ge 8\widetilde{K}$
  and $\rho_3\le\frac{\kappa}{4\sqrt{2}\lambda}$ where $\widetilde{K}$
  is the constant in Lemma \ref{mulnoise}, then for all $k\in\mathbb{N}$
  the inequality \eqref{err-ineq1} holds w.p. at least $1-Cp^2\exp(-cns^{-2}\zeta^{-1})$
 where $C,c$ are universal positive constants and $\zeta$ is a positive
 function on $\beta^*,\tau^2,\sigma^2,\kappa,\lambda_{\rm min}(\Sigma_M^\dag)$
 and the constants in \eqref{const}.
 \end{corollary}
 %-----------------------------------------------------------------Corollary
 \begin{corollary}\label{corollary-sign-mul}
  Suppose that $\Sigma$ satisfies the $\kappa$-REC one the set $\mathcal{C}(S^{*})$.
  Write
  $\widetilde{K}'=\widetilde{K}\max(2,1\!+\!s\|\widetilde{G}_{\rm mul}\|_{\rm max})$
  and $\widetilde{K}''=\widetilde{K}s\|[\widetilde{\Sigma}_{\rm mul}]_{\!S^*S^*}^{-1}\|_{\rm max}$
  where $\widetilde{K}$ is same as in Lemma \ref{mulnoise}.
  If the parameters $\lambda,\rho_1$ and $\rho_3$ in Algorithm \ref{Alg} are chosen such that
  $\lambda\ge 6\widetilde{K}'$, $\rho_1\!>\max\!\big(\frac{4a}{(a+1)\min_{i\in S^*}\!|\beta^{*}_{i}\!|},
  \frac{5\kappa \widetilde{K}''}{8\lambda})$ and $\rho_3\le\sqrt{\frac{5\kappa}{18\sqrt{3}\lambda}}$,
  then the result of Corollary \ref{corollary-sign} holds
  w.p. at least $1-Cp^2\exp(-cns^{-2}\zeta^{-1})$,
  where $C$ and $c$ are universal positive constants and $\zeta$ is
  a positive function depending on $\beta^*,\tau^2,\sigma^2,\kappa,\lambda_{\rm min}(\Sigma_M^\dag)$
 and the constants in \eqref{const}.
 \end{corollary}

  \noindent
  {\bf\large Appendix C}

  In this part we pay our attention to the implementation of GEP-MSCRA.
  From Section \ref{sec3}, we know that GEP-MSCRA consists of solving
  a sequence of weighted $\ell_1$-regularized LS,
  which can be equivalently written as
  \begin{equation}\label{Eprob-sub}
   \min_{\beta,u\in\mathbb{R}^p}\Big\{\frac{1}{2}\|u\|^2
   +{\textstyle\sum_{i=1}^{m}}\omega_i|\beta_i|:\ \widetilde{Z}\beta-u=\widetilde{y}\Big\},
  \end{equation}
  where $\omega_i\!=n\lambda(1-w_i^k)$ for $i=1,\ldots,p$ are the weights.
  There are some solvers developed for \eqref{Eprob-sub}; for example,
  the {\bf SLEP} developed by \cite{LiuYe09} with
  the accelerated proximal gradient method in \cite{Nesterov07},
  and the semismooth Newton ALM developed by \cite{LiSunToh16}.
  Motivated by the performance of the semismooth Newton ALM of \cite{LiSunToh16},
  we apply it for solving the dual of \eqref{Eprob-sub}, i.e.,
  \begin{equation}\label{Edprob-sub}
   \min_{\zeta,\eta\in\mathbb{R}^p}\left\{\frac{1}{2}\|\zeta\|^2
   +\langle\widetilde{y},\zeta\rangle+\delta_{\Lambda}(\eta):\ \widetilde{Z}^{\mathbb{T}}\zeta-\eta=0\right\}
   \ \ {\rm with}\ \ \Lambda=[-\omega,\omega].
  \end{equation}
  For a given $\mu>0$, define the augmented Lagrangian function of \eqref{Edprob-sub} by
  \begin{equation*}
   L_\mu(\zeta,\eta;\beta)
   \!:=\frac{1}{2}\|\zeta\|^2
   +\langle\widetilde{y},\zeta\rangle+\delta_{\Lambda}(\eta)
    +\langle \beta,\widetilde{Z}^{\mathbb{T}}\zeta-\eta\rangle
    +\frac{\mu}{2}\|\widetilde{Z}^{\mathbb{T}}\zeta-\eta\|^2.
  \end{equation*}
  The iteration steps of the ALM for solving \eqref{Edprob-sub}
  are described as follows.
%---------------------------------------------------------------------------------------
 \begin{algorithm}[!h]%[!hbp]
 \caption{\label{SNAL}{\bf\ \ An inexact ALM for the dual problem \eqref{Edprob-sub}}}
 \textbf{Initialization:} Choose $\mu_0>0$ and a starting point $(\zeta^0,\eta^0,\beta^0)$. Set $j=0$.\\
 \textbf{while} the stopping conditions are not satisfied \textbf{do}
 \begin{enumerate}
  \item  Solve the following nonsmooth convex minimization inexactly
         \begin{equation}\label{ALM-subprob}
          (\zeta^{j+1},\eta^{j+1})
          \approx\mathop{\arg\min}_{\zeta,\eta\in\mathbb{R}^{p}}
                   L_{\mu_{\!j}}(\zeta,\eta;\beta^j).
          \end{equation}
  \item Update the multiplier by the formula
         \(
           \beta^{j+1}=\beta^j+\mu_{\!j}(\widetilde{Z}^{\mathbb{T}}\zeta^{j+1}-\eta^{j+1}).
         \)
  \item Update $\mu_{j+1}\uparrow\mu_\infty\leq\infty$. Set~$j\leftarrow j+1$, and then go to Step 1.
 \end{enumerate}
 \textbf{end while}
 \end{algorithm}	

  Next we focus on the solution of the subproblem \eqref{ALM-subprob}.
  For any $\zeta\in\mathbb{R}^p$, define
  $\Phi_{j}(\zeta):=\min_{\eta\in\mathbb{R}^p}L_{\mu_{\!j}}(\zeta,\eta;\beta^j)$.
  After an elementary calculation,
  \[
    \Phi_{j}(\zeta)=\frac{\mu_{\!j}}{2}
    \big\|\Pi_{\Lambda}\big(\widetilde{Z}^{\mathbb{T}}\zeta\!+\!{\beta^j}/{\mu_{\!j}}\big)
     \!-\!\big(\widetilde{Z}^{\mathbb{T}}\zeta\!+\!{\beta^j}/{\mu_{\!j}}\big)\big\|^2
     \!+\!\frac{1}{2}\|\zeta\|^2\!+\!\langle\widetilde{y},\zeta\rangle.
  \]
  It is easy to verify that $(\zeta^{j+1},\eta^{j+1})$ is an optimal
  solution of \eqref{ALM-subprob} iff
  \[
    \zeta^{j+1}=\mathop{\arg\min}_{\zeta\in\mathbb{R}^p}\Phi_{j}(\zeta)\ \ {\rm and}\ \
    \eta^{j+1}=\Pi_{\Lambda}\big(\widetilde{Z}^{\mathbb{T}}\zeta^{j+1}+{\beta^j}/{\mu_{\!j}}\big).
  \]
  By the strong convexity of $\Phi_{j}$,
  $\zeta^{j+1}=\mathop{\arg\min}_{\zeta\in\mathbb{R}^p}\Phi_{j}(\zeta)$
  iff $\zeta^{j+1}$ satisfies
  \begin{equation}\label{nonsmooth-system}
    \nabla\Phi_{j}(\zeta)=\widetilde{y}+\zeta+\mu_{\!j} \widetilde{Z}\left[\!\Big(\widetilde{Z}^{\mathbb{T}}\zeta\!+\!{\beta^j}/{\mu_{\!j}}\Big)
    -\Pi_{\Lambda}\Big(\widetilde{Z}^{\mathbb{T}}\zeta\!+\!{\beta^j}/{\mu_{\!j}}\Big)\right]=0.
  \end{equation}
  The system \eqref{nonsmooth-system} is strongly semismooth (see
  the related discussion in \citep{Mifflin77,QiSun93}),
  and we apply the semismooth Newton method for solving it.
  Write $h:=\widetilde{Z}^{\mathbb{T}}\zeta\!+\!{\beta^j}/{\mu_{\!j}}$.
  By Proposition 2.3.3 and Theorem 2.6.6 of \citet{FH83}, the Clarke Jacobian
  $\partial\nabla\Phi_{j}$ satisfies
  \begin{equation}\label{inclusion}
    \partial(\nabla\Phi_{j})(\zeta)\subseteq \widehat{\partial}^2\Phi_{j}(\zeta)
    := I + \mu_j\widetilde{Z}\big(I-\partial\Pi_{\Lambda}(h)\big)\widetilde{Z}^{\mathbb{T}}
  \end{equation}
  where $\widehat{\partial}^2\Phi_{j}$ is the generalized Hessian of $\Phi_{j}$ at $\zeta$.
  Since the exact characterization of $\partial\nabla\Phi_{j}$
  is difficult to obtain, we replace $\partial\nabla\Phi_{j}$ with $\widehat{\partial}^2\Phi_{j}$
  in the solution of \eqref{nonsmooth-system}. Let $W\in\partial\Pi_{\Lambda}(h)$.
  By Theorem 2.6.6 of \cite{FH83}, $W={\rm Diag}(\varpi_1,\ldots,\varpi_p)$
  with $\varpi_i\in\partial\Pi_{\Lambda_i}(h_i)$ where
  \[
    \partial\Pi_{\Lambda_i}(h_i)
    =\left\{\begin{array}{cl}
     \{1\} &\ {\rm if}\ |h_i|<\omega_i;\\
    \ [0,1]&\ {\rm if}\ |h_i|=\omega_i;\\
    \ \{0\}&\ {\rm if}\ |h_i|>\omega_i.
    \end{array}\right.
  \]
  From the last two equations, each element in $\widehat{\partial}^2\Phi_{j}(\zeta)$
  is positive definite, which by \cite{QiSun93} implies that
  the following semismooth Newton method has a fast convergence rate.
%-------------------------------------------------------------------------------------
 \begin{algorithm}[!h]
 \caption{\label{SNCG}{\bf\ \ A semismooth Newton-CG algorithm for \eqref{nonsmooth-system}}}
 \textbf{Initialization:} Choose $\vartheta,\varsigma,\delta\!\in(0,1)$, $\varrho\in\!(0,\frac{1}{2})$
               and $\zeta^0\in\!\mathbb{R}^p$. Set $l=0$.\\
 \textbf{while} the stopping conditions are not satisfied \textbf{do}
 \begin{enumerate}
  \item  Choose a matrix $V^l\in\widehat{\partial}^2\Phi_{j}(\zeta^l)$.
               Solve the following linear system
               \begin{equation}\label{SNCG-dj}
                 V^ld=-\nabla\Phi_{j}(\zeta^l)
              \end{equation}
              with the conjugate gradient (CG) algorithm to find $d^l$ such that
              \[
                \|V^ld^l+\nabla\Phi_{j}(\zeta^l)\|\le\min(\vartheta,\|\nabla\Phi_{j}(\zeta^l)\|^{1+\varsigma}).
              \]
  \item Set $\alpha_l=\delta^{m_l}$, where $m_l$ is the first nonnegative integer $m$ for which
               \begin{equation*}
                \Phi_{j}(\zeta^l+\delta^md^l)\leq\Phi_{j}(\zeta^l)+\varrho\delta^m\langle\nabla\Phi_{j}(\zeta^l),d^l\rangle.
               \end{equation*}

  \item Set $\zeta^{l+1}=\zeta^l+\alpha_ld^l$ and $l\leftarrow l+1$, and then go to Step 1.
  \end{enumerate}
  \textbf{end while}
  \end{algorithm}

  It is worthwhile to point out that due to the special structure of $V^l$,
  the computation work of solving the linear system \eqref{SNCG-dj} is tiny;
  see the discussion in Section 3.3 of \cite{LiSunToh16}.
  During the implementation of the semismooth Newton ALM,
  we terminated the iterates of Algorithm \ref{SNAL} when
  \(
    \max\{\epsilon_{{\rm pinf}}^j,\epsilon_{{\rm dinf}}^j,\epsilon_{{\rm gap}}^j\}\le \epsilon^j,
  \)
  where $\epsilon_{{\rm gap}}^j$ is the primal-dual gap, i.e., the sum of the objective values
  of \eqref{Eprob-sub} and \eqref{Edprob-sub} at $(\beta^{j},\zeta^{j},\eta^{j})$,
  and $\epsilon_{{\rm pinf}}^j$ and $\epsilon_{{\rm dinf}}^j$ are the primal and
  dual infeasibility measure at $(\beta^{j},\zeta^{j},\eta^{j})$.
  By comparing the optimality condition of \eqref{ALM-subprob} with
  that of \eqref{Edprob-sub}, we defined
  \[
    \epsilon_{{\rm pinf}}^j
    :=\frac{\|\nabla\Phi_j(\zeta^j)\|}{1+\|\widetilde{y}\|}
    \ \ {\rm and}\ \
    \epsilon_{{\rm dinf}}^{j}:=\frac{\|\beta^{j}-\beta^{j-1}\|}{\mu_{\!j-1}(1+\|\widetilde{y}\|)}.
  \]
  We adopted a stopping criteria similar to those in \cite{LiSunToh16}:
  \[
    \|\nabla\Phi_j(\zeta^{j+1})\|\le \delta_j\min\big(0.1,\max(\epsilon_{{\rm dinf}}^j,\epsilon_{{\rm gap}}^j)\big)
    \ \ {\rm with}\ \ {\textstyle\sum_{j=0}^\infty}\delta_j<\infty.
  \]

 \noindent
 {\bf\large Appendix D}

 This part includes our implementation for CoCoLasso. When the optimal solution
 $\overline{\Sigma}$ of \eqref{Convex1} is available, one may apply the semismooth
 Newton ALM in Appendix B for solving \eqref{CoCoLasso}. Therefore, we here focus on
 the computation of $\overline{\Sigma}$. The problem \eqref{Convex1}
 can be equivalently written as
 \begin{equation}\label{EConvex1}
   \min_{W,B\in\mathbb{S}^p}\Big\{\|B\|_{\rm max}\!: W-B=\widehat{\Sigma},\,W\succeq\widehat{\epsilon}I\Big\},
 \end{equation}
 whose dual, after an elementary calculation, takes the following form
 \begin{equation}\label{DEConvex1}
   \min_{Y\in\mathbb{S}_{+}^p\cap\mathbb{B}}\big\langle Y,\widehat{\Sigma}-\widehat{\epsilon}I\big\rangle
   \ \ {\rm with}\ \ \mathbb{B}:=\big\{Y\in\mathbb{S}^p\!: \|Y\|_1\le 1\big\}.
 \end{equation}
 Here, $\|Y\|_1$ means the elementwise $\ell_1$-norm of $Y$.
 Different from \cite{Datta16}, we use the ADMM with a large step-size
 $\tau\in(1,\frac{\sqrt{5}+1}{2})$ instead of the unit one
 to solve \eqref{EConvex1}. From the numerical results in \cite{SunYangToh15},
 the ADMM with a larger step-size has better performance. For a given $\mu\!>0$,
 define the augmented Lagrangian function of \eqref{EConvex1} by
 \begin{equation*}
  L_\mu(W,B;\Gamma)\!:=\|B\|_{\rm max}+\langle W-B-\widehat{\Sigma},\Gamma\rangle
    +({\mu}/{2})\|W-B-\widehat{\Sigma}\|_F^2.
 \end{equation*}
 The iterations of the ADMM for \eqref{EConvex1} with a step-size
 are as follows.
 %---------------------------------------------------------------------------------------
 \begin{algorithm}[!h]
 \caption{\label{ADMM-alg}{\bf\ \ ADMM for solving the problem \eqref{EConvex1}}}
 \textbf{Initialization:} Choose $\mu>0,\tau\in(1,\frac{\sqrt{5}+1}{2})$ and
  $(W^0,B^0,\Gamma^0)$. Set $k=0$.\\
 \textbf{while} the stopping conditions are not satisfied \textbf{do}
 \begin{enumerate}
  \item  Compute the following strongly convex minimization problem
         \begin{equation}\label{ADMM-subprob1}
          W^{k+1}=\mathop{\arg\min}_{W\succeq\widehat{\epsilon} I}L_{\mu}(W,B^k;\Gamma^k).
          \end{equation}
  \item  Compute the following strongly convex minimization problem
         \begin{equation}\label{ADMM-subprob2}
          B^{k+1}=\mathop{\arg\min}_{B\in\mathbb{S}^{p}}L_{\mu}(W^{k+1},B;\Gamma^k).
         \end{equation}
  \item Update the multiplier by the formula
        \[
           \Gamma^{k+1}=\Gamma^k+\tau\mu(W^{k+1}-B^{k+1}-\widehat{\Sigma}).
         \]
  \item Set~$k\leftarrow k+1$, and then go to Step 1.
 \end{enumerate}
 \textbf{end while}
 \end{algorithm}

 Due to the speciality of the constraint $W-B=\widehat{\Sigma}$,
 the convergence of Algorithm \ref{ADMM-alg} can be directly obtained
 from Theorem B.1 of \cite{FPST13} with $S=T=0$.
 By the expression of $L_\mu(W,B;\Gamma)$, it holds that
 \begin{align}
   W^{k+1}=\widehat{\epsilon}I + \Pi_{\mathbb{S}_{+}^n}\big(B^k-\mu^{-1}\Gamma^k+\widehat{\Sigma}-\widehat{\epsilon}I\big),\qquad\nonumber\\
  \label{proj-ball}
   B^{k+1}=(W^{k+1}+\mu^{-1}\Gamma^k-\widehat{\Sigma})
         -\Pi_{\mu^{-1}\mathbb{B}}\big(W^{k+1}+\mu^{-1}\Gamma^k-\widehat{\Sigma}\big)
 \end{align}
 where the equality \eqref{proj-ball} is obtained from
  \(
    {\rm prox}_{f^*}(G)+{\rm prox}_f(G)=G
  \)
  with ${\rm prox}_f(G):=\mathop{\arg\min}_{B\in\mathbb{S}^{p}}\big\{\frac{1}{2}\|B-G\|_F^2+f(B)\big\}$
  for $f(B):=\mu^{-1}\|B\|_{\rm max}$.
  Just like \cite{Datta16}, we use the algorithm proposed in \cite{Duchi08}
  to compute the projection involved in \eqref{proj-ball}.

 During our implementation of Algorithm \ref{ADMM-alg},
 we adjust $\mu$ dynamically by the ratio of
 the primal and dual infeasibility.
 By the optimality conditions of \eqref{EConvex1} and
 \eqref{ADMM-subprob1}-\eqref{ADMM-subprob2}, we measure
 the primal and dual infeasibility and the dual gap
 at $(W^{k+1},B^{k+1},\Gamma^{k+1})$ in terms of
 $\epsilon_{{\rm pinf}}^k,\epsilon_{{\rm dinf}}^{k}$ and $\epsilon_{{\rm gap}}^{k}$,
 where
 \begin{align*}
   \epsilon_{{\rm pinf}}^k
   :=\frac{\|\mu(B^{k+1}-B^k)+(\tau^{-1}\!-1)(\Gamma^{k+1}\!-\Gamma^k)\|_F}{1+\|\widehat{\Sigma}\|_F},\qquad\qquad\\
   \epsilon_{{\rm dinf}}^{k}:=\frac{\|\Gamma^{k+1}-\Gamma^{k}\|_F}{\tau\mu(1+\|\widehat{\Sigma}\|_F)}\ \
   {\rm and}\ \
  \epsilon_{{\rm gap}}^{k}\!:=\!\frac{|\|B^{k+1}\|_{\rm max}+\langle\Gamma^{k+1},\widehat{\Sigma}-\widehat{\epsilon}I\rangle|}
  {\max(1,0.5(|\Gamma^{k+1}|+|\langle\Gamma^{k+1},\widehat{\Sigma}-\widehat{\epsilon}I\rangle|))}.
 \end{align*}
 %Figure \ref{ADMM} shows that Algorithm \ref{ADMM-alg} with $\tau=1.618$
% and a dynamic $\mu$ improves the performance of Algorithm 1 in \cite{Datta16}.
% We use
% \(
%  \max\{\epsilon_{{\rm pinf}}^k,\epsilon_{{\rm dinf}}^k,\epsilon_{{\rm gap}}^k\}\le {\bf tol}
% \)
% as the stopping rule where $\epsilon_{{\rm gap}}^k$ is the primal-dual gap
% at $(W^{k+1},B^{k+1},\Gamma^{k+1})$, since the output $(W^{f},B^{f},\Gamma^{f})$
% under this rule with a small {\bf tol} is an approximate primal and dual
% optimal solution.
%%-------------------------- Please insert Figure 5.1-------------------------------------
% \begin{figure}[!h]
% \setlength{\abovecaptionskip}{1pt}
% \begin{center}
%  {\includegraphics[width=1.0\textwidth]{figure_ADMM.eps}}\\
%  \caption{\small The computing time and objective value output of Algorithm \ref{ADMM-alg}
%  with $\tau=1.618$ and Algorithm 1 in \cite{Datta16} respectively for solving \eqref{Convex1},
%  where $\widehat{\Sigma}$ comes from Example \ref{example1}.}
%   \label{ADMM}
% \end{center}
% \end{figure}

 \par
%%%%%%%%%%%%%%%%%%%%%%%%%%%%%%%%%%%%%%%%%%%%%%%%%%%%%%%%%%%%%%%%%%%%%%%%%%%%%%%%%%%%%%%%%%%%%%%%%%%%%%%%%%%%%%%%%%%%%%%%%%%%
\vskip 14pt
\noindent {\large\bf Acknowledgements}

 The authors would like to express their sincere thanks to anonymous
 referees for valuable suggestions and comments for the original manuscript.
 The authors are deeply indebted to Professor Po-Ling Loh for sharing R
 and Matlab codes for computing the NCL estimator. The research of Shaohua Pan
 and Shujun Bi is supported by the National Natural Science Foundation of China under project
 No.11571120 and No.11701186.

 \par

%%%%%%%%%%%%%%%%%%%%%%%%%%%%%%%%%%%%%%%%%%%%%%%%%%%%%%%%%%%%%%%%%%%%%%%%%%%%%%%%%%%%%%%%%%%%%%%%%%%%%%%%%%%%%%%%%%%%%%%%%%%%
\markboth{\hfill{\footnotesize\rm Ting Tao, Shaohua Pan AND Shujun Bi} \hfill}
{\hfill {\footnotesize\rm CALIBRATED ZERO-NORM REGULARIZED ESTIMATOR} \hfill}

%\iffalse
\bibhang=1.7pc
\bibsep=2pt
\fontsize{9}{14pt plus.8pt minus .6pt}\selectfont
\renewcommand\bibname{\large \bf References}
%\begin{thebibliography}{11}
\expandafter\ifx\csname
natexlab\endcsname\relax\def\natexlab#1{#1}\fi
\expandafter\ifx\csname url\endcsname\relax
  \def\url#1{\texttt{#1}}\fi
\expandafter\ifx\csname urlprefix\endcsname\relax\def\urlprefix{URL}\fi
%\fi

%%%%%%%%%%%%%%%%%%%%%%%%%%%%%%%%%%%%%%%%%%%%%%%%%%%%%%%%%%

\newpage
%%%%%%%%%%%%%%%%%%%%%%%%%%%%%%%%%%%%%%%%%%%%%%%%%%%%%%%%%%%%%%%%%%%%%%%%%%%%%%%%%%%%%%%%%%%%%%%%%%%%%%%%%%%%%%%%%%%%%%%%%%%%
\vskip .65cm
\noindent
 School of Mathematics, South China University of Technology
\vskip 2pt
\noindent
E-mail: (\em 201620122022@mail.scut.edu.cn)
\vskip 2pt

\noindent
School of Mathematics, South China University of Technology
\vskip 2pt
\noindent
E-mail: (shhpan@scut.edu.cn)

\noindent
School of Mathematics, South China University of Technology
\vskip 2pt
\noindent
E-mail: (bishj@scut.edu.cn)

% \vskip .3cm
%\centerline{(Received ???? 20??; accepted ???? 20??)}\par
\end{document}